\documentclass[10pt]{amsart}
\usepackage{amsmath}
\usepackage{amsfonts}
\usepackage{amssymb} 
\usepackage{amsthm}
\usepackage{url} 
\usepackage{bbm}
\usepackage{xcolor}
\usepackage{multicol}
\usepackage{pst-fractal}
\usepackage{comment}
\usepackage{hyperref}
\usepackage[backend=biber, style=alphabetic, maxbibnames=99]{biblatex}
\usepackage[T1]{fontenc}
\usepackage{newpxtext,newpxmath} %compatibility with palatino font
\usepackage[utf8]{inputenc}
		    \usepackage{palatino}

\usepackage{tikz-cd}
\usepackage{verbatim}
\usepackage{physics}

\hypersetup{colorlinks=true,
linkcolor=blue, citecolor=blue, urlcolor=blue}
\usepackage{fullpage} 
\usepackage[margin=1in]{geometry}

\newtheorem{thm}{Theorem}[section]

\newtheorem{cor}[thm]{Corollary}
\newtheorem{lem}[thm]{Lemma}
\newtheorem{prop}[thm]{Proposition}

\theoremstyle{definition}
\newtheorem{defn}[thm]{Definition}

\newtheorem{rmk}[thm]{Remark}

\newtheorem{ex}[thm]{Example}

\DeclareMathOperator{\Span}{span}

\DeclareMathOperator{\MAX}{MAX}

\newcommand{\N}{\mathbb{N}}

\newcommand{\e}{\epsilon}

\renewcommand{\a}{\alpha}
\renewcommand{\b}{\beta}

\newcommand{\vp}{\varphi}

\renewcommand{\norm}[1]{\left\lVert #1 \right\rVert}

\newcommand{\bb}[1]{\mathbb{#1}}
 
\newcommand{\cc}[1]{\mathcal{#1}}

\title{Operator Systems Generated by Projections}
\author{Roy Araiza}
\author{Travis Russell}

\address{Department of Mathematics \& Illinois Quantum Information Science and Technology Center, University of Illinois at Urbana-Champaign, Urbana, IL, 61801}
\email{raraiza@illinois.edu}

\address{Department of Mathematics, Dartmouth College, Hanover, NH}
\email{travis.b.russell@tcu.edu}

\begin{document}
\maketitle
\newcommand{\rmkra}[1]{%
{\textcolor{red}{[Roy: \small #1]}}
}

\newcommand{\rmktr}[1]{%
{\textcolor{red}{[Travis: \small #1]}}
}
\begin{abstract}
We construct a family of operator systems and $k$-AOU spaces generated by a finite number of projections satisfying a set of linear relations. This family is universal in the sense that the map sending the generating projections to any other set of projections which satisfy the same relations is completely positive. These operator systems are constructed as inductive limits of explicitly defined operator systems. By choosing the linear relations to be the nonsignalling relations from quantum correlation theory, we obtain a hierarchy of ordered vector spaces dual to the hierarchy of quantum correlation sets. By considering another set of relations, we also find a new necessary condition for the existence of SIC-POVM's.
\end{abstract}

\section{Introduction}

The interplay between operator algebras and quantum information theory has yielded many exciting results, especially over past fifteen years or so, including in the recent solution to Connes' embedding problem \cite{ji2020mip}. Within this field of study, operator systems and completely positive maps have played a vital role, providing crucial operator algebraic tools used to approach quantum information problems. For example, the various correlation sets considered in Tsirelson's problems \cite{Tsirelson1987, Tsirelson1993} have been reformulated in terms of states on finite dimensional operator systems which arise as subsystems of certain universal group C*-algebras in \cite{lupini2020perfect}.

The connection between operator systems and quantum information theory is very natural. Quantum measurements are generally formulated in terms of projection-valued measures (or positive operator-valued measures) and states on the C*-algebras which they generate. Therefore many problems in quantum information theory have a natural formulation in terms of projection-valued measures and states on the corresponding C*-algebra generated by those projections. Since many problems involve only discrete measurements, and thus only finitely many projections, it is reasonable to wonder if these problems can be formulated using only the langauge of finite-dimensional operator systems and their state spaces. One impediment to this is that operator systems, as abstractly characterized by Choi-Effros \cite{choi1977injectivity}, can generate a variety of non-isomorphic C*-covers --- C*-algebras generated by completely order-isomorphic copies of the given operator system. Key properties, such as $p=p^2$ for a projection $p$, or $ef=fe$ for a pair of commuting operators $e$ and $f$, can be forgotten in certain C*-covers of an operator system. Therefore operator systems considered in the quantum information literature often arise as subsystems of specific C*-covers instead of being defined as abstract operator systems without a specified C*-cover (c.f. \cite{lupini2020perfect}).

In recent work \cite{araiza2020abstract}, the authors abstractly characterized the elements of an operator system which arise as projections in the corresponding C*-envelope (the canonical ``smallest'' C*-cover). In joint work with Tomforde \cite{araiza2021universal,araiza2021matricial}, the authors used these notions to characterize quantum and quantum commuting correlations entirely in the language of abstract operator systems and their states. These results demonstrate that abstract operator systems have a sufficiently rich theory to capture problems in quantum information without reference to ambient C*-algebras.

In this paper, we generalize previous work, particularly the results of \cite{araiza2021universal}, to construct a family of universal operator systems each spanned by a finite set of projections $\{p_1, p_2, \dots, p_N\}$ satisfying a finite set of linear relations. Provided there exists at least one family of projections $\{P_1, P_2, \dots, P_N\}$ in $B(H)$ satisfying a set of relations $\cc R$, then a universal operator system $\cc U_{\cc R}$ exists and has the following properties:
\begin{enumerate}
\item $\cc U_{\cc R}$ is spanned by its unit $e$ and positive elements $p_1, \dots, p_N$ satisfy the relations in $\cc R$,
\item the elements $p_1, \dots, p_N$ are projections in the C*-envelope of $\cc U_{\cc R}$, and
\item if $Q_1, \dots, Q_N$ are projections on a Hilbert space $K$ which, together with $I_K$, satisfy the relations in $\cc R$, then the mapping defined by $e \mapsto I_K$ and $p_i \mapsto Q_i$ is completely positive.
\end{enumerate}
The operator system $\cc U_{\cc R}$ is constructed from elementary ingredients as an inductive limit of operator systems, without reference to ``concrete'' operator systems arising from known C*-algebras. Furthermore, for each integer $k \in \mathbb{N}$, we construct another universal operator system $\cc U_{\cc R}^k$ which is $k$-minimal and satisfies a similar universal property in the category of $k$-minimal operator systems (or $k$-AOU spaces, in the language of \cite{araiza2021matricial}). Operator systems which are $k$-minimal were first studied by Xhabli in \cite{xhabli2012super}, where they are realized as operator subsystems of direct sums of matrix algebras of size no greater than $k \times k$. The universal operator system $\cc U_{\cc R}^k$ has the property that the map $e \mapsto I$ and $p_i \mapsto Q_i$ is completely positive whenever $\text{span} \{I, Q_1, Q_2, \dots, Q_N\} \subseteq B(H)$ is $k$-minimal (for example, this occurs when $\dim(H) \leq k$).

To demonstrate the usefulness of the universal operator systems $\cc U_{\cc R}$ and $\cc U_{\cc R}^k$ above, we study two problems in quantum information theory: Tsirelson's problems on correlation sets and Zauner's conjecture on SIC-POVMs. Each of these problems can be formulated in terms of projections on Hilbert spaces satisfying certain relations, allowing us to make use of the operator system $\cc U_{\cc R}$. Furthermore, it is important to distinguish the case when the dimension of the Hilbert space is constrained to be finite in both of these applications, allowing us to make use of the $k$-minimal operator system $\cc U_{\cc R}^k$.

For quantum correlations, we show that by choosing input-output parameters $n,m \in \mathbb{N}$ and choosing the relations $\cc R$ to be the non-signalling conditions, we recover a hierarchy of AOU spaces 
\[ V_{loc}(n,m), V_{qa}(n,m), V_{qc}(n,m), V_{ns}(n,m) \]
which is dual, in the sense of Kadison duality \cite{kadison1951representation}, to the hierarchy of quantum correlation sets $C_{loc}(n,m) \subseteq C_{qa}(n,m) \subseteq C_{qc}(n,m) \subseteq C_{ns}(n,m)$. The positive cone of each AOU space in the hierarchy is constructed as an inductive limit of cones. Thus it could be possible to distinguish the various correlation sets by studying the inductive limits involved in the definition of these AOU spaces. Since Connes' embedding problem is equivalent to asking if $V_{qa}(n,m) = V_{qc}(n,m)$ for all input-output parameters $n,m \in \mathbb{N}$, our constructions yield a potentially new path for approaching this problem.

For Zauner's conjecture, we devise new necessary conditions for the existence of a SIC-POVM in the matrix algebra $M_d$. A SIC-POVM is a family of rank one projections $P_1, P_2, \dots, P_{d^2}$ which satisfy the relation $\sum P_i = d I$ and $\Tr(P_i P_j) = \frac{1}{d+1}$ whenever $i \neq j$. It was conjectured by Zauner that a SIC-POVM exists in every dimesion $d$. However, this conjecture has only been verified for finitely many values of $d$ \cite{Fuchs2017SICs}. We study this problem by considering the universal $d$-minimal operator system $\cc U_{\cc R}^d = \text{span} \{e, p_1, p_2, \dots, p_{d^2}\}$ satisfying the single relation $\sum p_i = de$ where $e$ is unit of $\cc U_{\cc R}$. Whenever a SIC-POVM $\{P_1, P_2, \dots, P_{d^2}\}$ exists in $M_d$, it follows from the universal properties of $\cc U_{\cc R}$ that the mapping $\pi: \cc U_{\cc R} \to M_d$ defined by $\pi(p_i) = P_i$ is unital and completely positive. Using this observation, we uncover necessary conditions on the operator system $\cc U_{\cc R}$ which must hold whenever a SIC-POVM exists. We conclude the paper with a remark on how a similar approach gives rise to necessary conditions for the existence of families of $d+1$ mutually unbiased bases in $\mathbb{C}^d$, another important open problem in quantum information theory \cite{Raynal2011MUBs}.

The paper is organized as follows. In Section 2, we introduce notation and provide preliminary details on operator systems, $k$-minimality, and abstract projections. In Sections 3 and 4, we develop the universal operator system $\cc U_{\cc R}$. In Section 5, we develop the universal $k$-minimal operator systems $\cc U_{\cc R}^k$. In Section 6, we consider applications to quantum correlation sets, and in Section 7 we consider applications to SIC-POVMs and mutually unbiased bases.

\section{Preliminaries} \label{sec: preliminaries}

In this section, we recall some basic facts from the theory of operator systems as well as preliminary results on $k$-AOU spaces and projections in operator systems. We begin by mentioning the notation used in this paper. We let $\mathbb{N}, \mathbb{R}$, and $\mathbb{C}$ denote the sets of natural numbers, real numbers, and complex numbers, respectively. For each $n \in \mathbb{N}$, we let $[n] := \{1,2,\dots,n\}$. For each $n,k \in \mathbb{N}$, we let $M_{n,k}$ denote the set of $n \times k$ matrices with entries in $\mathbb{C}$, and we let $M_n := M_{n,n}$. For each $n \in \mathbb{N}$, we let $M_n^+$ denote the cone of positive semidefinite matrices. We let $M_{n,k}(\mathbb{R})$ denote the set of $n \times k$ matrices with entries in $\mathbb{R}$. Given matrices $A \in M_n$ and $B \in M_k$, we let $A \otimes B$ denote the Kronecker tensor product.

\subsection{Operator systems and completely positive maps}

A \textbf{$*$-vector space} is a complex vector space $\mathcal V$ together with an conjugate-linear involution $*: \mathcal{V} \to \mathcal{V}$. An element $x \in \mathcal V$ such that $x^* = x$ is called \textbf{hermitian} and we denote the real subspace of all hermitian elements of $\mathcal{V}$ by $\mathcal V_h$.
If $\mathcal{V}$ is a $*$-vector space, a \textbf{cone} is a subset $C \subseteq \mathcal{V}_h$ with $\alpha C \subseteq C$ for all $\alpha \in [0, \infty)$ and such that $C + C \subseteq C.$ We will say the cone $C$ is \textbf{proper} if $C \cap -C = \{ 0 \}$.
An \textbf{ordered $*$-vector space} $(\mathcal{V},C)$ consists of a $*$-vector space $\mathcal{V}$ with a proper cone $C$.  For any ordered vector space $(\mathcal{V},C)$ we may define a partial order on $\mathcal{V}_h$ by $v \leq w$ (equivalently $w \geq v$) if and only if $w-v \in C$. 
If $(\mathcal{V},C)$ is an ordered $*$-vector space, an element $e \in \mathcal{V}_h$ is called an \textbf{order unit} if for all $v \in \mathcal{V}_h$ there exists $r > 0$ such that $re \geq v$.  An order unit $e$ is called \textbf{Archimedean} if whenever $re+v \geq 0$ for all real $r >0$, then $v \geq 0$.  An \textbf{Archimedean order unit space} (or \textbf{AOU space} for short) is a triple $(\mathcal{V},C,e)$ such that $(\mathcal{V},C)$ is an ordered $*$-vector space and $e$ is an Archimedean order unit for $(\mathcal{V},C)$. 
If $e$ is an order unit for $(\mathcal{V},C)$, the \textbf{Archimedean closure} of $C$ is defined to be the set of $x \in \mathcal{V}_h$ with the property that $re + x \in C$ for all $r > 0$. In general the Archimedean closure of a proper cone $C$ may not be proper.
If $\mathcal{V}$ is a complex vector space, then for any $n \in \mathbb{N}$ the vector space of $n \times n$ matrices with entries in $\mathcal{V}$ is denoted $M_n(\mathcal{V})$.  We see that $M_n(\mathcal{V})$ inherits a $*$-operation by $(a_{i,j})_{i,j}^* = (a_{j,i}^*)_{i,j}$. Let $\mathcal{V}$ be a $*$-vector space.  A family of \textbf{matrix cones} $\{ \mathcal C_n \}_{n=1}^\infty$ is a collection such that $\mathcal C_n$ is a proper cone of $M_n(\mathcal{V})$ for all $n \in \mathbb{N}$.  We call a family of matrix cones $\{\mathcal C_n \}_{n=1}^\infty$ a \textbf{matrix ordering} if $$\alpha^* \mathcal C_n \alpha \subseteq \mathcal C_m$$ for all $\alpha \in M_{n,m} (\mathbb{C})$. We often use a calligraphic symbol such as $\mathcal{C}$ to denote a matrix ordering; i.e. $\mathcal{C} := \{\mathcal C_n\}_{n=1}^\infty$. When $\mathcal{C}$ is a matrix ordering, we let $\mathcal{C}_n$ denote the $n\textsuperscript{th}$ matrix cone of the matrix ordering.
If $x \in \mathcal{V}$, for every $n \in \mathbb{N}$ we define $$x_n := I_n \otimes x = \left( \begin{smallmatrix} x & &  \\ & \ddots & \\  & & x \end{smallmatrix} \right) \in M_n(\mathcal{V}).$$  An \textbf{operator system} is a triple $(\mathcal{V}, \mathcal{C}, e)$ consisting of a $*$-vector space $\mathcal{V}$, a matrix ordering $\mathcal{C}$ on $\mathcal{V}$, and an element $e \in \mathcal{V}$ such that $(\mathcal{V}, \mathcal{C}_n, e_n)$ is an AOU space for all $n \in \mathbb{N}$. In this case, we call $e$ an \textbf{Archimedean matrix order unit}. If we only have that $e$ is an order unit for each $(\mathcal{V}, \mathcal{C}_n)$, then we call $e$ a \textbf{matrix order unit}. We often let $\mathcal{V}$ denote the operator system $(\mathcal{V}, \mathcal{C}, e)$ when the unit and matrix ordering are unspecified or clear from context.

If $(\mathcal{V},C)$ and $(\mathcal W,D)$ are ordered $*$-vector spaces, a linear map $\phi : \mathcal{V} \to \mathcal{W}$ is called \textbf{positive} if $\phi(C) \subseteq D$.  A positive linear map $\phi : \mathcal{V} \to \mathcal{W}$ is an order isomorphism if $\phi$ is a bijection and $\phi(C) = D$. An injective map $\phi: \mathcal{V} \to \mathcal{W}$ is called an \textbf{order embedding} if it is an order isomorphism onto its range.
If $\mathcal{V}$ and $\mathcal{W}$ are $*$-vector spaces and $\phi : \mathcal{V} \to \mathcal{W}$ is a linear map, then for each $n \in \mathbb{N}$ the map $\phi$ induces a linear map $\phi_n : M_n(\mathcal{V}) \to M_n(\mathcal{W})$ by $\phi_n ( (a_{i,j})_{i,j}) = ( \phi(a_{i,j}) )_{i,j}$.  If $(\mathcal{V}, \mathcal{C}, e)$ and $(\mathcal{W}, \mathcal{D}, f)$ are operator systems, a linear map $\phi : \mathcal{V} \to \mathcal{W}$ is called \textbf{completely positive} if $\phi_n(\mathcal{C}_n) \subseteq \mathcal{D}_n$ for all $n \in \mathbb{N}$.  A completely positive $\phi : \mathcal{V} \to \mathcal{W}$ is called \textbf{unital} if $\phi(e) = f$.  A completely positive map $\phi : \mathcal{V} \to \mathcal{W}$ is called an  \textbf{complete order isomorphism} if $\phi$ is a bijection and $\phi (\mathcal{C}_n) = \mathcal{D}_n$ for all $n \in \mathbb{N}$. A linear map $\phi: \mathcal{V} \to \mathcal{W}$ is called a \textbf{complete order embedding} if $\phi$ is a complete order isomorphism onto its range. 

We now recall the representation theorem of Choi and Effros, and some consequences.

\begin{thm}[Choi-Effros, \cite{choi1977injectivity}] \label{thm: Choi-Effros}
Let $(\mathcal{V}, \mathcal{C}, e)$ be an operator system. Then there exists a Hilbert space $H$ and a unital complete order embedding $\pi: \mathcal{V} \to B(H)$.
\end{thm}

\noindent By Theorem \ref{thm: Choi-Effros}, every operator system arises as a subspace of $B(H)$, and hence every operator system generates a C*-algebra. However, that C*-algebra is not necessarily unique. By a \textbf{C*-cover}, we mean a pair $(\mathcal{A}, \pi)$ consisting of a C*-algebra $\mathcal{A}$ and a unital complete order embedding $\pi: \mathcal{V} \to \mathcal{A}$ such that $C^*(\pi(\mathcal{V})) = \mathcal{A}$. Among all C*-covers, there exist canonical ``smallest'' and ``largest'' ones. We will be concerned with the ``smallest'' C*-cover, called the \textit{C*-envelope}. A \textbf{C*-envelope} for an operator system $\mathcal{V}$ is a C*-cover, denoted $(C^*_e(\mathcal{V}), i)$, which satisfies the following universal property: if $(\mathcal{B},j)$ is another C*-cover, then the identity map $id: j(s) \mapsto i(s)$ from $j(\mathcal{V})$ to $i(\mathcal{V})$ extends uniquely to $*$-homomorphism $\pi: \mathcal{B} \to C^*_e(\mathcal{V})$. The following theorem asserts that every operator system has a C*-envelope, which is necessarily unique.

\begin{thm}[Hamana, \cite{hamana1979injective}] \label{thm: Hamana}
Let $\mathcal{V}$ be an operator system. Then there exists a C*-envelope $(C^*_e(\mathcal{V}), i)$ and it is unique up to $*$-isomorphism.
\end{thm}

\subsection{$k$-AOU spaces}

We will also be using facts regarding \textbf{k-AOU} spaces, which we will discuss briefly. The interested reader will find more details in \cite{araiza2021matricial}. 
\begin{defn}[$k$-Archimedean order unit space]\label{defn: k-aou space}
For any $k \in \mathbb{N}$, a \textbf{$k$-Archimedean order unit space} (or $k$-AOU space, for short) is a triple $( \mathcal{V}, C, e)$ consisting of
\begin{itemize}
\item[(i)] $\mathcal{V}$, a $*$-vector space, 
\item[(ii)] $C \subseteq M_k(\mathcal{V})_h$, a proper cone, \textbf{compatible} in the sense that for each $\alpha \in M_{k} (\mathbb{C})$, we have $\alpha^* C \alpha \subseteq C$, and
\item[(iii)] $e \in \mathcal{V}$ with the property that $e_k := I_k \otimes e$ is an Archimedean order unit for $(M_k(\mathcal{V}), C)$.  
\end{itemize}
A pair $(\mathcal{V}, C)$ satisfying conditions $(i)$ and $(ii)$ is called a \textbf{$k$-ordered $*$-vector space}, and an element $e$ satisfying condition $(iii)$ is called a \textbf{$k$-Archimedean order unit} for the $k$-ordered vector space $(\mathcal{V},C)$.
\end{defn}

Next, we define the appropriate morphisms in the category of $k$-AOU spaces.

\begin{defn}[$k$-positive maps]
Let $k \in \mathbb{N}$, and suppose $(\mathcal{V},C)$ and $(\mathcal{W},D)$ are $k$-ordered $*$-vector spaces. A linear map $\phi: \mathcal{V} \to \mathcal{W}$ is called \textbf{$k$-positive} if $\phi_k(C) \subseteq D$. If $\phi$ is $k$-positive and injective with $\phi_k^{-1}(D) \subseteq C$, then $\phi$ is called a \textbf{$k$-order embedding}. A bijective $k$-order embedding is called a \textbf{$k$-order isomorphism}.
\end{defn}

In the case when $k=1$, it is clear our notion of a $k$-AOU space is identical to that of an AOU space, and that $k$-positive maps, $k$-order embeddings, and $k$-order isomorphisms are just positive maps, order embeddings, and order isomorphisms, respectively.

In \cite{PaulsenTodorovTomfordeOpSysStructures}, a variety of operator system structures were considered for AOU spaces. In particular, the authors constructed minimal and maximal operator system structures (minimal and maximal with respect to inclusions of matricial orderings). We wish to consider these structures when the initial object is a $k$-AOU space. The following definitions come from \cite{araiza2021matricial} and \cite{xhabli2012super}.

\begin{defn}[Operator System Structure]
Let $k \in \mathbb{N}$, and suppose $(\mathcal{V}, C, e)$ is a $k$-AOU space. If $\mathcal{C}$ is an Archimedean closed matrix ordering on $\mathcal{V}$ satisfying $\mathcal{C}_k = C$, then we say $\mathcal{C}$ \textbf{extends} $C$ or is an \textbf{extension} of $C$, and we call the operator system $(\mathcal{V}, \mathcal{C}, e)$ an \textbf{operator system structure} on $(\mathcal{V}, C, e)$.
\end{defn}

We will now focus on two operator system structures on $k$-AOU spaces. 

\begin{defn}[The $k$-minimal operator system structure on a $k$-AOU space]
Given a $k$-AOU space $(\mathcal{V}, C, e)$, we define
\[ C_n^\text{$k$-min} := \{ x \in M_n(\mathcal{V})_h : \alpha^* x \alpha \in C \text{ for all } \alpha \in M_{n,k} \} \]
for each $n \in \mathbb{N}$. If $C^{\text{$k$-min}}:= \{C_n^{\text{$k$-min}}\},$ then the triple $(\cc V, C^{\text{$k$-min}}, e)$ is called a \textbf{$k$-minimal operator system}. 
\end{defn}

\begin{defn}
Given $k \in \bb N$, let $(\mathcal{V}, C, e)$ be a $k$-AOU space. For each $n \in \mathbb{N}$, define
\[ D_n^\text{$k$-max}(\mathcal{V}) := \{ \alpha^* \text{diag}(s_1, \dots, s_m) \alpha: \alpha \in M_{mk,n} \text{ and } s_1, \dots, s_m \in C, m \in \bb N \}, \]
and let $C_n^\text{$k$-max}$ denote the Archimedean closure of $D_n^\text{$k$-max}$. If $C^{\text{$k$-max}}:= \{C_n^{\text{$k$-max}}\},$ then the triple $(\cc V, C^{\text{$k$-max}}, e)$ is called a \textbf{$k$-maximal operator system}. 
\end{defn} 

While we will make use of the $k$-minimal operator system structure extensively, we will only consider the $k$-maximal structure in the case when $k=1$. The details of this case can be found in \cite{PaulsenTodorovTomfordeOpSysStructures}. For more details regarding both the $k$-minimal and $k$-maximal structures, we refer the reader to \cite{xhabli2012super} and \cite{araiza2021matricial}.

It is a quick exercise to verify that given linear map $\phi: \cc V \to \mathcal W$ in which $\cc V$ is a $k$-AOU space equipped with the $k$-max structure and $\cc W$ is an operator system, that $\phi$ is completely positive if and only if it is $k$-positive. Similarly, if $\phi: \cc W \to \cc V$ where $\cc W$ is an operator system and $\cc V$ is a $k$-minimal operator system, then $\phi$ is $k$-positive if and only if it is completely positive. The next theorem will be useful for us later.

\begin{thm}[\cite{xhabli2012super} and \cite{araiza2021matricial}] \label{thm: k-min characterization}
Let $(\mathcal{V}, C, e)$ be a $k$-AOU space and let $n \in \mathbb{N}$. Then the following statements are equivalent:
\begin{enumerate}
    \item $x \in C^{k-\text{min}}_n$.
    \item $\varphi_n(x) \in M_{nk}^+$ for every $k$-positive map $\varphi: \mathcal{V} \to M_k$.
    \item $\varphi_n(x) \in M_{nk}^+$ for every unital $k$-positive map $\varphi: \mathcal{V} \to M_{k}$.
\end{enumerate}
In particular, the map
\[ \pi := \bigoplus_{\varphi \in \mathfrak{S}} \varphi \]
is a unital complete order embedding on $(\mathcal{V}, C^{k-\text{min}}, e)$, where $\mathfrak{S}$ denotes the set of all unital $k$-postive maps from $\mathcal{V}$ to $M_k$.
\end{thm}

\subsection{Abstract projections}

Suppose that $\mathcal{V}$ is an operator system, $p \in \mathcal{V}$, and that $j(p)$ is a projection in some C*-cover $(\mathcal{A},j)$ for $\mathcal{V}$. By Theorem \ref{thm: Hamana}, there exists a $*$-homomorphism $\pi: \mathcal{A} \to C^*_e(\mathcal{V})$ such that $\pi(j(p)) = i(p)$. Since the image of a projection under a $*$-homomorphism is a projection, we conclude that the image of $p$ in $C^*_e(\mathcal{V})$ is a projection.

By an \textbf{abstract projection}, we mean an element $p$ in an operator system $\mathcal{V}$ whose image in $C^*_e(\mathcal{V})$ is a projection. Equivalently, by the above argument, an abstract projection is an element whose image is a projection in some C*-cover.

Abstract projections can be characterized intrinsically in terms of the matrix ordering and order unit. To do so, we first define a cone which will be useful throughout the paper. In the following, and throughout the paper, we let $J_n$ denote $n \times n$ matrix whose entries are all equal to 1.

\begin{defn} \label{defn: C(p)}
Let $(\mathcal{V}, \mathcal{C}, e)$ be an operator system, and suppose that $p \in \mathcal{V}_h$ and $0 \leq p \leq e$. Then $\mathcal{C}(p) = \{ \mathcal{C}(p)_n\}_{n=1}^\infty$ denotes the matrix ordering defined by: $x \in \mathcal{C}(p)_n$ if and only if for every $\epsilon > 0$ there exists $t > 0$ such that
\[ x \otimes J_2 + \epsilon I_n \otimes (p \oplus p^{\perp}) + t I_n \otimes (p^{\perp} \oplus p) \in \mathcal{C}_{2n} \]
where $p^{\perp} := e - p$.
\end{defn}

\begin{rmk}
In \cite{araiza2021universal}, the notation $\mathcal{C}(p)$ denoted a matrix ordering on the vector space $M_2(\mathcal{V})$ rather than a matrix ordering on $\mathcal{V}$. However, intersecting the matrix ordering from \cite{araiza2021universal} with the image of $\mathcal{V}$ under the complete order embedding $x \mapsto x \otimes J_2$ yields the matrix ordering defined above.
\end{rmk}

For a given $p \in \mathcal{V}$ satisfying the conditions of Definition \ref{defn: C(p)}, it may not be the case that $\mathcal{C}(p)$ is a proper matrix ordering. However, it is always the case that $\mathcal{C}_n \subseteq \mathcal{C}(p)_n$. Indeed, if $x \in \mathcal{C}_n$, then
\[ x \otimes J_2 + \epsilon I_n \otimes (p \oplus p^{\perp}) + \epsilon I_n \otimes (p^{\perp} \oplus p) = x \otimes J_2 + \epsilon I_{2n} \otimes e \in \mathcal{C}_{2n}. \]
When $\mathcal{C}(p)$ happens to be a proper matrix ordering, the following holds.

\begin{prop}[{\cite[Lemma 3.6]{araiza2021universal}}] \label{prop: p projection in C(p)}
Let $(\mathcal{V}, \mathcal{C}, e)$ be an operator system. Suppose $0 \leq p \leq e$. If $\mathcal{C}(p)$ is proper, then $p$ is an abstract projection in $(\mathcal{V}, \mathcal{C}(p), e)$.
\end{prop}

The case when $\mathcal{C}_n = \mathcal{C}(p)_n$ is of special importance.

\begin{thm}[{\cite[Theorem 5.10]{araiza2020abstract}}] \label{thm: abstract projections}
Let $(\mathcal{V}, \mathcal{C}, e)$ be an operator system, and suppose that $p \in \mathcal{V}_h$ and $0 \leq p \leq e$. Then $p$ is an abstract projection if and only if $\mathcal{C} = \mathcal{C}(p)$.
\end{thm}

\section{Universal operator systems generated by contractions}\label{sec: universal relation vector space}

In this section, we will construct a vector space $\mathcal{U}_{\mathcal{R}}$ spanned by generators $e, p_1, \dots, p_N$ satisfying a finite set of relations $\mathcal{R}$ and a matrix ordering $\mathcal{C}$ such that $(\mathcal{U}_{\mathcal{R}}, \mathcal{C}, e)$ is an operator system that is universal with respect to all families of positive contractions satisfying the relations $\mathcal{R}$. This means that if $q_1,\dots,q_N \in B(H)$ are positive contractions and the operators $\{I_H, q_1, \dots, q_N\}$ satisfy the relations $\mathcal{R}$, then the mapping $e \mapsto I_H, p_i \mapsto q_i$ is completely positive. Many of these results follow from  \cite{PaulsenTodorovTomfordeOpSysStructures} and \cite{araiza2021universal}, however they provide a foundational first step in the universal constructions presented in later sections.

We begin by constructing the vector space $\mathcal{U}_{\mathcal{R}}$. Suppose $\mathcal{R}$ consists of linear equations $\{r_1, r_2, \dots, r_l\}$ with real coefficients in the variables $\{e, p_1, \dots, p_N\}$. Then there exists a scalar matrix $M_{\mathcal{R}} \in M_{l,N+1}(\mathbb{R})$ such that $r_i$ corresponds to the $i^{\text{th}}$ row of the matrix vector equation $M_{\mathcal{R}} \begin{pmatrix} e & p_1 & \dots & p_N \end{pmatrix}^T = 0_l$ where $0_l$ is the zero vector in $\mathbb R^l$. More generally, we say that vectors $\{f, q_1, \dots, q_N\}$ \textbf{satisfy} $\mathcal{R}$ (or that $\mathcal V = \text{span} \{f, q_1, \dots, q_N\}$ \textbf{satisfies} $\mathcal{R}$ when the vectors $f,q_1, \dots, q_N$ are clear from context) if $M_\mathcal{R} \begin{pmatrix} f & q_1 & \dots & q_N \end{pmatrix}^T = 0_{l}$. 

In the following definition, we regard $M_{\mathcal{R}}$ as a linear map from $\mathbb{C}^{N+1}$ to $\mathbb{C}^l$ via matrix-vector multiplication.

\begin{defn}
Let $\mathcal{R}$ be a set of $l$ linear equations in variables $\{e,p_1, \dots, p_N\}$ expressed in the matrix-vector equation $M_{\mathcal{R}} \begin{pmatrix} e & p_1 & \dots & p_N \end{pmatrix}^T = 0_l$ where $M_{\mathcal{R}} \in M_{l,N+1}(\mathbb{R})$. We define the vector space $\mathcal{U}_{\mathcal{R}}$ to be the quotient vector space $\mathbb{C}^{N+1} / J_{\mathcal{R}}$ where $J_{\mathcal{R}} = \text{span} \{ r_i^T : i \in [l]\} \subseteq \mathbb{C}^{N+1}$ where $r_1, r_2, \dots r_l \in M_{1,N+1}(\mathbb{R})$ denote the rows of $M_{\mathcal{R}}$. Writing $e_0, e_1, \dots, e_N$ for the canonical basis vectors in $\mathbb{C}^{N+1}$, we define $e, p_1, \dots, p_N \in \mathcal{U}_{\mathcal{R}}$ by setting $e := e_0 + J_{\mathcal{R}}$ and $p_i := e_i + J_{\mathcal{R}}$ for each $i \in [N]$.
\end{defn}

It is evident that $\dim(\mathcal{U}_{\mathcal{R}}) = N + 1 - \rank(M_{\mathcal{R}})$. We now give a brief example and then describe the universal properties of $\mathcal{U}_{\mathcal{R}}$ as a vector space.

\begin{ex}
Suppose $\mathcal{R}$ consists of the single relation $\sum_{i=1}^N p_i = I$. Then we may take $M_{\mathcal{R}}$ to be the row matrix $\begin{pmatrix} 1 & -1 & -1 & \dots & -1 \end{pmatrix} \in M_{1,N+1}$. The rank of this matrix is 1, so the dimension of $\mathbb{C}^{N+1}/J_{\mathcal{R}}$ is $(N+1)-1 = N$.
\end{ex}

\begin{prop} \label{prop: universal vector space}
Suppose there exists a vector space $\mathcal{V}$ spanned by vectors $f, q_1, \dots, q_N$ satisfying $\mathcal{R}$. Then the map $\phi: \mathcal{U}_{\mathcal{R}} \to \mathcal{V}$ given by $\phi(e)=f$ and $\phi(p_i) = q_i$ is a well-defined linear map. Moreover, $\phi$ is injective if and only if $\dim(\mathcal{V}) = \dim(\mathcal{U}_{\mathcal{R}})$.
\end{prop}

\begin{proof}
First observe that the map $\widehat{\phi}: \mathbb{C}^{N+1} \to \mathcal{V}$ given by $\widehat{\phi}(e_0) = f$ and $\widehat{\phi}(e_i) = q_i$ for $i \in [N]$ is well defined and linear. We will show that $\widehat{\phi}(x) = 0$ for all $x \in J_{\mathcal{R}}$. This will imply that $\phi: \mathcal{U}_{\mathcal{R}} \to \mathcal{V}$ is well-defined and linear. Let $x \in J_{\mathcal{R}}$. Then $x = \sum a_i r_i^T$ and hence $\widehat{\phi}(x) = \sum a_i \widehat{\phi}(r_i^T)$. Suppose $r_i^T = \sum_{k=0}^N b_k e_k$. Then $\widehat{\phi}(r_i^T) = b_0 f + \sum_{k=1}^N b_k q_k$. Since $\{f, q_1, \dots, q_N\}$ satisfy $\mathcal{R}$, $b_0 f + \sum_{k=1}^N b_k f = 0$. So $\widehat{\phi}(r_i^T) = 0$. It follows that $\widehat{\phi}(x) = 0$. So $\phi$ is well-defined and linear. That $\phi$ is injective if and only if $\dim(\cc V) = \dim(\cc U_{\cc R})$ is clear. \end{proof}

Next, we will endow $\mathcal{U}_{\mathcal{R}}$ with the structure of an AOU space. To do this, we first need to show that $\mathcal{U}_{\mathcal{R}}$ is a $*$-vector space. This can be done by first identifying the vector space $\mathbb{C}^{N+1}$ with the $(N+1) \times (N+1)$ diagonal matrices and then identifying the canonical basis vectors $\{e_0, e_1, \dots, e_N\}$ with the diagonal matrices with a single non-zero entry of 1 in the $i^{\text{th}}$ diagonal component. Since $\mathcal{U}_{\mathcal{R}} = \mathbb{C}^{N+1} / J_{\mathcal{R}}$, the space $\mathcal{U}_{\mathcal{R}}$ inherits the involution from $\mathbb{C}^{N+1}$ since $J_{\mathcal{R}}$ is self-adjoint (i.e. spanned by vectors with real entries). 

\begin{defn}\label{defn: contraction cone}
Let $\mathcal{R}$ be a set of relations on generators $\{e, p_1, \dots, p_N\}$ and let $\mathcal{U}_{\mathcal{R}} = \text{span} \{e, p_1, \dots, p_N\}$ be the corresponding $*$-vector space. Let $p_i^{\perp} := e - p_i$. We define 
\[ \widetilde{E} := \text{cone}(e, p_1, \dots, p_N, p_1^{\perp}, \dots, p_N^{\perp}) = \{ r_0 e + \sum_{i=1}^n t_i p_i + \sum_{u=1}^n s_i p_i^{\perp} : r_0, t_1, \dots, t_N, s_1, \dots, s_N \in \mathbb{R}^+ \} \]
where $\mathbb{R}^+$ is the set of non-negative real numbers. We define $E$ to be the archimedean closure of $\widetilde{E}$ with respect to $e$, i.e
\[ E := \{ x \in \mathcal{U}_{\mathcal{R}} : x=x^* \text{ and for every } \epsilon > 0, x + \epsilon e \in \widetilde{E} \}. \]
\end{defn}

The requirement that both $p_i$ and $p_i^{\perp}$ are elements of $E$ ensures that $0 \leq p_i \leq e$ for each $i \in [N]$, so that each $p_i$ is a positive contraction. It also ensures that $e = p_i + p_i^{\perp}$ is an element of $E$. In many cases, the terms $p_i^{\perp}$ can be omitted in the definition of the cone $E$. For example, if $\cc R$ enforces the relation $\sum p_i = e$, then $\text{cone}(p_1,\dots,p_N)$ contains $e$ as well as $p_i^{\perp}$, since $\sum p_i = e$ implies that $p_i \leq e$ for each $i$.

The next result shows that if there exists some AOU space of dimension $\dim (\mathcal{U}_{\mathcal{R}})$ spanned by its unit and some positive contractions satisfying $\mathcal{R}$, then $(\mathcal{U}_{\mathcal{R}}, E, e)$ is an AOU space. Moreover, $\mathcal{U}_{\mathcal{R}}$ satisfies a universal property as an AOU space.

\begin{prop} \label{prop: Universal AOU space}
Suppose there exists an AOU space $\mathcal{V}$ with unit $f$ and positive contractions $q_1, \dots, q_N$ satisfying relations $\mathcal{R}$. Then the map $\phi: \mathcal{U}_{\mathcal{R}} \to \mathcal{V}$ given by $\phi(e) = f$ and $\phi(p_i) = q_i$ is a well-defined positive linear map. Moreover, if $\dim(\mathcal{V}) = \dim(\mathcal{U}_{\mathcal{R}})$, then $\phi$ is injective and $(\mathcal{U}_{\mathcal{R}}, E, e)$ is an AOU space (i.e. $E$ is a proper cone).
\end{prop}

\begin{proof}
Let $D$ denote the positive cone of the AOU space $\mathcal{V}$. Since each $q_i$ is a positive contraction, we have $\text{cone}(q_1,\dots, q_N, q_1^{\perp}, \dots, q_N^{\perp}) \subseteq D$, where $q_i^{\perp} := f - q_i$. Since $\phi(\widetilde{E}) = \text{cone}(q_1,\dots, q_N, q_1^{\perp}, \dots, q_N^{\perp})$, $\phi$ is positive on $\widetilde{E}$. If $x \in E$, then $x + \epsilon e \in \widetilde{E}$ for every $\epsilon > 0$. Hence $\phi(x) + \epsilon f \in D$ for every $\epsilon > 0$. Since $f$ is an archimedean order unit, $\phi(x) \in D$. So $\phi$ is positive on $E$.

Now suppose that $\dim(\mathcal{V}) = \dim(\mathcal{U}_{\mathcal{R}})$. Then $\phi$ is injective, since $\mathcal{V}$ is the range of $\phi$. Since $(\mathcal{V}, D, f)$ is an AOU space, the cone $D$ is proper. Positivity and injectivity of $\phi$ then imply $x = 0$. Therefore $E$ is proper. It follows that $(\mathcal{U}_{\mathcal{R}}, E, e)$ is an AOU space. \qedhere
\end{proof}

We conclude this section by endowing the AOU space $(\mathcal{U}_{\mathcal{R}}, E, e)$ with the maximal operator system structure, as defined in Section~\ref{sec: preliminaries}.  We begin by recalling the following: 

\begin{thm}[Theorem 3.22 of \cite{PaulsenTodorovTomfordeOpSysStructures}] \label{thm: cmax universal property}
Let $(\mathcal V,C,e)$ be an AOU space. Then $(\mathcal V,C^{\text{max}},e)$ is an operator system. If $(\mathcal W, \mathcal{D}, e)$ is an operator system and $\varphi: \mathcal V \to \mathcal W$ is a positive map, then $\varphi$ is completely positive on $(\mathcal V,C^{\text{max}},e)$.
\end{thm}

We can regard the map sending an AOU space $(\mathcal V,C,e)$ to the operator system $(\mathcal V, C^{\text{max}},e)$ as a functor from the category of AOU spaces to the category of operator systems. This functor carries positive maps between AOU spaces to completely positive maps between operator systems. Our final result of this section follows easily from the properties of this functor.

\begin{thm} \label{thm: Universal contraction system}
Let $H$ be a Hilbert space. Suppose that there exist positive operators $q_1, \dots, q_N \in B(H)$ such that the vectors $\{I_H, q_1, \dots, q_N\}$ satisfy the relations $\mathcal{R}$. Then the map $\varphi: \mathcal{U}_{\mathcal{R}} \to B(H)$ defined by $\varphi(e) = I_H, \varphi(p_i) = q_i$ is completely positive on $(\mathcal{U}_{\mathcal{R}}, E^{\text{max}}, e)$. Moreover, if $\dim(\text{span}\{I_H, q_1, \dots, q_N\}) = \dim(\mathcal U_{\mathcal{R}})$ then $(\mathcal{U}_{\mathcal{R}}, E^{\text{max}}, e)$ is an operator system (i.e. the matrix ordering is proper).
\end{thm}

\begin{proof}
Let $\mathcal{V}$ denote the operator space spanned by $I_h, q_1, \dots q_N$ in $B(H)$. By Proposition \ref{prop: Universal AOU space}, $\varphi: \mathcal{U}_{\mathcal{R}} \to \mathcal{V}$ is positive on $E$. By Theorem \ref{thm: cmax universal property}, $\varphi$ is completely positive on $E^{\text{max}}$. If $\dim(\mathcal{V}) = \dim(\mathcal{U}_{\mathcal{R}})$, then by Proposition \ref{prop: Universal AOU space} $\varphi$ is injective and $(\mathcal{U}_{\mathcal{R}}, E, e)$ is an AOU space. It follows from Theorem \ref{thm: cmax universal property} again that $(\mathcal{U}_{\mathcal{R}}, E^{\text{max}}, e)$ is an operator system in this case.
\end{proof}

\section{Universal operator systems spanned by projections}\label{sec: univeral projection system}

For our main results, we will need to construct matrix orderings $\mathcal{C}$ which satisfy $\mathcal{C} = \mathcal{C}(p)$ for specified elements $p$ in a vector space. This will be accomplished using inductive limits of matrix orderings. 

\begin{defn}
Let $\mathcal{V}$ be a $*$-vector space and let $e \in \mathcal{V}_h$. A \textbf{nested increasing sequence of matrix orderings} is a sequence $\mathcal{C}^1, \mathcal{C}^2, \dots$ of matrix orderings on $\mathcal{V}$ for which $\mathcal{C}_k^{n} \subseteq \mathcal{C}_k^{n+1}$ for all $n,k \in \mathbb{N}$ and for which $(\mathcal{V}, \mathcal{C}^n, e)$ is an operator system for all $n \in \mathbb{N}$. The \textbf{inductive limit} of a nested increasing sequence of matrix orderings $\{\mathcal{C}^n\}_{n=1}^\infty$ is the sequence $\mathcal{C}^{\infty} := \{ \mathcal{C}_k^{\infty}\}_{k=1}^\infty$ defined by: $x \in \mathcal{C}_k^{\infty}$ if and only if for every $\epsilon > 0$ there exists $n \in \mathbb{N}$ such that $x + \epsilon I_k \otimes e \in \mathcal{C}_k^n$.
\end{defn}

More concisely, the inductive limit of a nested increasing sequence of matrix orderings is the Archimedean closure of the union of the matrix orderings. An inductive limit of matrix orderings is not necessarily proper. However, we can say the following.

\begin{thm}[\cite{araiza2021universal}]
Let $\{ \mathcal{C}^n \}$ be a nested increasing sequence of matrix orderings on a $*$-vector space $\mathcal{V}$ with unit $e$. If the inductive limit $\mathcal{C}^{\infty}$ is proper, then $(\mathcal{V}, \mathcal{C}^{\infty}, e)$ is an operator system.
\end{thm}

For more general results on inductive limits of operator systems, see \cite{MawhinneyTodorov2017}. We will make use of the following result which relates abstract projections and inductive limits.

\begin{prop} \label{prop: inductive limit projection}
Let $\{ \mathcal{C}^n \}$ be a nested increasing sequence of matrix orderings on a $*$-vector space $\mathcal{V}$ with unit $e$ and suppose that $p \in \mathcal{V}_h$. If $p$ is an abstract projection for $(\mathcal{V}, \mathcal{C}^n, e)$ for every $n \in \mathbb{N}$, and if $\mathcal{C}^{\infty}$ is proper, then $p$ is an abstract projection for $(\mathcal{V}, \mathcal{C}^{\infty}, e)$.
\end{prop}

\begin{proof}
Suppose that $x \in \mathcal{C}^\infty (p)_k$. Let $\epsilon > 0$. Then there exists $t > \epsilon$ such that
\[ x \otimes J_2 + \frac{\epsilon}{2} I_k \otimes (p \oplus p^{\perp}) + t I_k \otimes (p^{\perp} \oplus p) \in \mathcal{C}^{\infty}_{2k}. \]
By the definition of $\mathcal{C}^\infty_{2k}$ there exists $n \in \mathbb{N}$ such that
\[ x \otimes J_2 + \frac{\epsilon}{2} I_k \otimes (p \oplus p^{\perp}) + t I_k \otimes (p^{\perp} \oplus p) + \frac{\epsilon}{2} I_{2k} \otimes e \in \mathcal{C}^{n}_{2k}. \]
Let $\delta > 0$. Choose $r > 0$ such that
\[ \begin{pmatrix} \delta & \epsilon \\ \epsilon & r \end{pmatrix} \geq 0. \]
Then
\[ \begin{pmatrix} \delta & \epsilon \\ \epsilon & r \end{pmatrix} \otimes p \otimes I_k  =  \begin{pmatrix} 0 & \epsilon \\ \epsilon & 0 \end{pmatrix} \otimes p \otimes I_k + \begin{pmatrix} \delta & 0 \\ 0 & r \end{pmatrix} \otimes p \otimes I_k \in \mathcal{C}_{2k}^n \]
and
\[ \begin{pmatrix} r & \epsilon \\ \epsilon & \delta \end{pmatrix} \otimes p^{\perp} \otimes I_k =  \begin{pmatrix} 0 & \epsilon \\ \epsilon & 0 \end{pmatrix} \otimes p^{\perp} \otimes I_k + \begin{pmatrix} r & 0 \\ 0 & \delta \end{pmatrix} \otimes p^{\perp} \otimes I_k \in \mathcal{C}_{2k}^n. \]
Summing these terms, we have
\[ \epsilon \begin{pmatrix} 0 & e \\ e & 0 \end{pmatrix} \otimes I_k  + \delta (p \oplus p^{\perp}) \otimes I_k  + r (p^{\perp} \oplus p) \otimes I_k \in \mathcal{C}_{2k}^n. \]
Applying the canonical shuffle  $M_2(\mathcal{V}) \otimes M_k \to M_k \otimes M_2(\mathcal{V})$ to the above matrix, we obtain
\[ \alpha := \epsilon I_k \otimes \begin{pmatrix} 0 & e \\ e & 0 \end{pmatrix}  + \delta I_k \otimes (p \oplus p^{\perp}) + r I_k \otimes (p^{\perp} \oplus p) \in \mathcal{C}_{2k}^n. \]
Therefore for every $\delta > 0$ there exists $s  = t - \epsilon/2 + r > 0$ such that
\begin{eqnarray}
\mathcal{C}^n_{2k} & \ni & x \otimes J_2 + \frac{\epsilon}{2} I_k \otimes (p \oplus p^{\perp}) + t I_k \otimes (p^{\perp} \oplus p) + \frac{\epsilon}{2} I_{2k} \otimes e + \alpha \nonumber \\
& = & (x + \epsilon I_k \otimes e) \otimes J_2 + \delta I_k \otimes (p \oplus p^{\perp}) + (t - \frac{\epsilon}{2} + r) I_k \otimes (p^{\perp} \oplus p). \nonumber
\end{eqnarray}
We conclude that $x + \epsilon I_k \otimes e \in \mathcal{C}_k^n$, since $p$ is abstract projection in $(\mathcal{V}, \mathcal{C}^n, e)$. Thus, for every $\epsilon > 0$, there exists $n \in \mathbb{N}$ such that $x + \epsilon I_k \otimes e \in \mathcal{C}_k^n$. It follows that $x \in \mathcal{C}^\infty_k$. Since $x$ was arbitrary, we conclude that $p$ is an abstract projection in $(\mathcal{V}, \mathcal{C}^{\infty}, e)$.
\end{proof}

The above proposition, together with the preliminary results on abstract projections, will provide the basic ingredients for constructing universal operator systems spanned by projections. Let $p_1, \dots, p_N$ be the positive contractions satisfying $\mathcal{R}$ generating the operator system $(\mathcal{U}_{\mathcal{R}}, E^{\text{max}}, e)$. In the following, we will construct a new matrix ordering $\mathcal{E}^{\text{proj}}$ for $\mathcal{U}_{\mathcal{R}}$ as an inductive limit of matrix orderings. The first term of the inductive limit will be $\mathcal{E}^0 := E^{\text{max}}$. To obtain the rest of the sequence, first extend the list $\{p_1, \dots, p_N\}$ to an infinite sequence by setting $p_{m} = p_i$ whenever $m = i \mod N$ for every $m \in \mathbb{N}$ and $i \in [N]$. For example, $p_{N+1} = p_1, p_{N+2}=p_2$, and so on. With this convention, we define $\mathcal{E}^{n+1} = \mathcal{E}^k(p_{n+1})$ (in the notation of Definition \ref{defn: C(p)}) for every $n \in \mathbb{N}$. We define $\mathcal{E}^{\text{proj}} := \mathcal{E}^{\infty}$.

\begin{thm}\label{thm: universal projection system}
Suppose there exists an operator system $(\mathcal{V}, \mathcal{D}, f)$ spanned by projections $q_1, \dots, q_N$ and unit $f$ satisfying relations $\mathcal{R}$ with $\dim(\mathcal{V}) = \dim(\mathcal{U}_{\mathcal{R}})$. Then: 
\begin{enumerate}
    \item the matrix ordering $\mathcal{E}^{\text{proj}}$ is proper.
    \item the elements $p_1, \dots, p_N$ in the operator system $(\mathcal{U}_{\mathcal{R}}, \mathcal{E}^{\text{proj}}, e)$ are abstract projections.
    \item the mapping $\varphi: \mathcal{U}_{\mathcal{R}} \to \mathcal{V}$ defined by $\varphi(e)=f$ and $\varphi(p_i) = q_i$ for all $i \in [N]$ is completely positive.
\end{enumerate}
\end{thm}

\begin{proof}
We first check (3). By Proposition \ref{prop: universal vector space}, the mapping $\varphi$ is a linear bijection since the elements $q_1, \dots, q_N$ satisfy $\mathcal{R}$. We will show inductively that for each $n \in \mathbb{N}$, $\varphi$ is completely positive with respect to the matrix ordering $\mathcal{E}^n$. When $n=0$, $\varphi$ is completely positive by Theorem \ref{thm: Universal contraction system}. Suppose that $\varphi$ is completely positive on $\mathcal{E}^n$. Let $p = p_{n+1}$, so that $\mathcal{E}^{n+1} = \mathcal{E}^n(p)$. Let $q = q_i$ where $i = n+1 \mod N$. Then $\varphi(p)=q$. Suppose that $x \in \mathcal{E}^{n+1}_m$. Then for every $\epsilon > 0$ there exists $t > 0$ such that
\[ x \otimes J_2 + \epsilon I_m \otimes (p \oplus p^{\perp}) + t I_m \otimes (p^{\perp} \oplus p) \in \mathcal{E}^n_{2m}. \]
Since $\varphi$ is completely positive on $\mathcal{E}^n$, for every $\epsilon > 0$ there exists $t > 0$ such that
\[ \varphi_{m}(x) \otimes J_2 + \epsilon I_m \otimes (q \oplus q^{\perp}) + t I_m \otimes (q^{\perp} \oplus q) \in \mathcal{D}_{2m}. \]
Since $q$ is a projection, this implies that $\varphi_{m}(x)$ is positive. Therefore $\varphi$ is completely positive on $\mathcal{E}^{n+1}$. By induction, $\varphi$ is completely positive on $\mathcal{E}^n$ for every $n$ and hence it is completely positive on $\mathcal{E}^{\infty} = \mathcal{E}^{\text{proj}}$. This proves statement (3).

For statement (1), since $\varphi$ is a completely positive injective map and since $\varphi(\mathcal{E}_1^{\text{proj}} \cap -\mathcal{E}_1^{\text{proj}}) = \{0\}$, we conclude that $\mathcal{E}^{\text{proj}}$ is proper. For statement (2), since $\mathcal{E}^{\text{proj}}$ is proper we see that $(\mathcal{U}_{\mathcal{R}}, \mathcal{E}^{\text{proj}}, e)$ is an operator system. Let $i \in [N]$. To see that $p_i$ is an abstract projection, observe that $\mathcal{E}^{\text{proj}}$ is the inductive limit of the sequence $\{\mathcal{E}^{kN + i}\}_{k=0}^\infty$. Since each term in this sequence has the form $\mathcal{E}^{n}(p_i)$ for some $n \in \mathbb{N}$, $p_i$ is an abstract projection in each term of the inductive limit of operator systems, by Proposition \ref{prop: p projection in C(p)}. By Proposition \ref{prop: inductive limit projection}, $p_i$ is a projection in $\mathcal{E}^{\text{proj}}$. This proves statement (2).
\end{proof}

\section{Universal $k$-AOU spaces spanned by projections}\label{sec: universal k-relation AOU}

In this section, we will endow the vector space $\mathcal{U}_{\mathcal{R}}$ with the structure of $k$-AOU space such that its generators $\{p_1, p_2, \dots, p_N\}$ are all abstract projections. This $k$-AOU space will be universal with respect to all $k$-AOU spaces spanned by projections satisfying $\mathcal{R}$. The ordering on this $k$-AOU space will be constructed as an inductive limit of cones, analogous to the inductive limit of matrix orderings used in the previous section. However, we will need new techniques to show that this inductive limit has the desired properties.

\begin{defn}
Given a $k$-AOU space $(\cc V, C, e)$ and a contraction $p \in C$, then define the set $C[p]$ to be the set of all $x \in M_k(\mathcal{V})_h$ such that for every $\e>0$ there exists $t>0$ satisfying \[
(\alpha + \beta)^* x (\alpha + \beta) + (\epsilon \alpha^* \alpha + t \beta^* \beta) \otimes p + (\epsilon \beta^* \beta + t \alpha^* \alpha) \otimes p^{\perp} \in C \]
for all $\alpha, \beta \in M_k$.
\end{defn}

\begin{lem}\label{lem: alternate characterization of C[p]}
Given a $k$-AOU space $(\cc V, C, e)$ with positive contraction $p$, then $x \in C[p]$ if and only if for every $\epsilon > 0$ there exists $t > 0$ such that
\[ \begin{pmatrix} x & x \\ x & x \end{pmatrix} + \epsilon \begin{pmatrix} p & 0 \\ 0 & p^{\perp} \end{pmatrix} \otimes I_k + t \begin{pmatrix}  p^{\perp} & 0 \\ 0 &  p \end{pmatrix} \otimes I_k \in C_{2k}^{k-\text{min}}. \]
\end{lem}

\begin{proof}
Let $x \in C[p]$. Then for all $\e >0$ there exists $t>0$ such that \begin{align*}
    (\alpha + \beta)^*x(\alpha + \beta) + (\epsilon \alpha^*\alpha + t\beta^*\beta) \otimes p + (t\alpha^*\alpha + \epsilon \beta^*\beta) \otimes p^\perp \in C,
\end{align*} for all $\alpha,\beta \in M_k$. We rewrite the above as \begin{align*}
    \begin{pmatrix}
    \alpha \\
    \beta
    \end{pmatrix}^* [ J_2 \otimes x + \epsilon (p \oplus p^\perp) \otimes I_k + t (p^\perp \oplus p) \otimes I_k] \begin{pmatrix}
    \alpha \\
    \beta 
    \end{pmatrix},
\end{align*} which implies $x \otimes J_2 + \epsilon I_k \otimes (p \oplus p^\perp) + tI_k \otimes (p^\perp \oplus p) \in C_{2k}^{k-\text{min}}$ by applying the canonical shuffle 
\[ \varphi: M_2 \otimes M_k \otimes \mathcal{V} \to M_k \otimes M_2 \otimes \mathcal{V}. \]
Conversely, if $x \otimes J_2 + \epsilon I_k \otimes (p \oplus p^\perp) + tI_k \otimes (p^\perp \oplus p) \in C_{2k}^{k-\text{min}}$ for some $\epsilon, t > 0$ then by applying the canonical shuffle $\varphi^{-1}$ and after conjugation by $\begin{pmatrix}
\alpha &
\beta
\end{pmatrix}^T$, compatibility of the k-minimal structure yields $(\alpha + \beta)^*x(\alpha + \beta) + (\epsilon \alpha^*\alpha + t\beta^*\beta) \otimes p + (t\alpha^*\alpha + \epsilon \beta^*\beta) \otimes p^\perp \in C_k^{k-\text{min}}$. Since $C^{k-\text{min}}$ is an extension of $C$, 
\[ (\alpha + \beta)^*x(\alpha + \beta) + (\epsilon \alpha^*\alpha + t\beta^*\beta) \otimes p + (t\alpha^*\alpha + \epsilon \beta^*\beta) \otimes p^\perp \in C. \]
Since this holds for all $\e >0$, we have the conclusion. \qedhere
\end{proof}

\begin{cor} \label{cor: C=C[p] means proj in k-min}
Let $k \in \mathbb{N}$, and let $(\mathcal{V},C,e)$ be a $k$-AOU space. Suppose that $p \in \mathcal{V}$ is a positive contraction. Then $p$ is an abstract projection in the operator system $(\mathcal{V}, C^{k-\text{min}}, e)$ if and only if $C = C[p]$.
\end{cor}

\begin{proof}
Applying the canonical shuffle $\varphi: M_2 \otimes M_k \otimes \mathcal{V} \to M_k \otimes M_2 \otimes \mathcal{V}$ to the expression in Lemma \ref{lem: alternate characterization of C[p]}, we see that $x \in C[p]$ if and only if $x \in C^{k-\text{min}}(p)_k$. For any $n \in \mathbb{N}$, $C^{k-\text{min}}(p)_n \subseteq C[p]_n^{k-\text{min}}$ since $C[p] = C^{k-\text{min}}(p)_k$. If $p$ is an abstract projection in $(\mathcal{V}, C^{k-\text{min}}, e)$, then $C[p] = C^{k-\text{min}}(p)_k = C$. On the other hand, if $C = C[p]$, then for every $n \in \mathbb{N}$, $C^{k-\text{min}}(p)_n \subseteq C[p]_n^{k-\text{min}} = C_n^{k-\text{min}}$, implying that $C^{k-\text{min}}(p)_n \subseteq C^{k-\text{min}}_n$. But we have $C^{k-\text{min}}_n \subseteq C^{k-\text{min}}(p)_n$ automatically, so $C^{k-\text{min}}(p) = C^{k-\text{min}}$. So $p$ is an abstract projection in $(\mathcal{V}, C^{k-\text{min}}, e)$.
\end{proof}

\begin{defn}
Let $(\mathcal{V}, C, e)$ be a $k$-AOU space. Then a positive contraction $p \in \mathcal{V}$ is called an \textbf{abstract projection in the $k$-AOU sense} if $C = C[p]$. When it is clear from context that $\mathcal{V}$ is a $k$-AOU space, we simply call $p$ an abstract projection in $\mathcal{V}$.
\end{defn}

Our goal will be to construct a sequence of cones in $M_k(\mathcal{U}_{\mathcal{R}})$ for which $\{p_1, p_2, \dots, p_N \}$ are all abstract projections in the $k$-AOU sense for the resulting $k$-AOU structure on $\mathcal{U}_{\mathcal{R}}$. Although the strategy is similar to the one used in Section \ref{sec: univeral projection system}, we will not be able to apply the results of that section directly. This is partly because we do not know that $p$ is an abstract projection in the $k$-AOU space $(\mathcal{V}, C[p], e)$, even when $C[p]$ is proper, whereas in Section \ref{sec: univeral projection system} we exploited the fact that $p$ was an abstract projection in the operator system $(\mathcal{V}, \mathcal{C}(p), e)$ whenever $\mathcal{C}(p)$ was proper (Proposition \ref{prop: p projection in C(p)}).

In the following, $\lim_{n \to \infty} C^n$ and $C^{\infty}$ both denote the Archimedean closure of the union of a nested increasing sequence of cones $\{C^n\}_{n \in \bb N}$, where each $C^n \subseteq M_k(\cc V)_h$ is a cone on the $k \times k$ matrices over the $*$-vector space $\cc V$. We call the resulting cone $C^{\infty}$ the \textbf{$k$-inductive limit} of the sequence $\{C^n\}$.

\begin{lem}\label{lem: inductive limits of k-cones form a non-proper k-aou}
Let $\cc V$ be a $*$-vector space. Let $\{C^n\}_{n \in \bb N}$ be a nested increasing sequence of proper cones $C^n \subseteq M_k(\cc V)_h$ such that for each $n \in \bb N$ the triple $(\cc V, C^n, e)$ is a $k$-AOU space.
 Then the $k$-inductive limit $C^\infty$ is a (possibly non-proper) cone, and the triple $(\cc V, C^\infty, e)$ forms a (possibly non-proper) k-AOU space. 
\end{lem}

\begin{proof}
The fact that $C^\infty$ is closed under sums and action by nonnegative real numbers is immediate. If $a \in M_k$ is non-zero, $x \in C^\infty,$ and $\e>0$, then let $L \in \bb N$ such that $x + \frac{\e}{\norm{a}^2} I_k \otimes e \in C^L$. Conjugating by $a$ and using
\[ a^*xa + \frac{\e}{\norm{a}^2} a^*a \otimes e \leq a^*xa + \e I_k \otimes e \]
it follows that $a^*xa + \e I_k \otimes e \in C^L$ (if $a=0$, this also holds trivially). Hence, for every $\epsilon > 0$ there exists $L \in \mathbb{N}$ such that $a^*xa + \e I_k \otimes e \in C^L$ and therefore $a^*xa \in C^\infty.$ 

Let $x \in M_k(\cc V)_h$ and let $L \in \bb N$. Then there exists $r>0$ such that $rI_k \otimes e-x \in C^L \subseteq C^\infty$. Thus, $e$ is an order unit for the pair $(\cc V, C^\infty)$. Similarly, suppose for all $\e >0$ one has $\e I_k \otimes e + x \in C^\infty$. Then for each $\epsilon > 0$ there exists $L \in \bb N$ such that $(\frac{\e}{2}I_k \otimes e + x) + \frac{\e}{2}I_k \otimes e \in C^L.$ Thus, for every $\e > 0$ there exists $L \in \mathbb{N}$ such that $\e I_k \otimes e + x \in C^L$. By the definition of $C^{\infty}$, $x \in C^{\infty}$. This finishes the proof. \qedhere
\end{proof}

For the remainder of this section, let $\mathfrak{S}(C)$ denote the set of unital $k$-positive maps $\phi: (\cc V, C, e) \to M_k$, where $(\cc V, C, e)$ is a (possibly non-proper) $k$-AOU space.

\begin{lem}\label{lem: states on inductive k-limit}
Let $\cc V$ be a $k$-AOU space and let $\{C^n\}_{n \in \bb N}$ be a nested increasing sequence of cones in $M_k(\cc V)_h$ such that $(\cc V, C^n, e)$ is a $k$-AOU space for each $n \in \mathbb{N}$. Then 
\[ \bigcap_{n=1}^\infty \mathfrak{S}(C^n) = \mathfrak{S}(C^\infty) \] 
\end{lem}

\begin{proof}
Let $u \in \mathfrak S(C^\infty)$ and consider $x \in C^L$ for some $L \in \bb N$. Since $C^L \subseteq \bigcup_{n =1}^\infty C^n \subseteq  C^\infty$, then we have $u_k(x) \in M_{k^2}^+$. Thus $\mathfrak S(C^\infty) \subseteq \bigcap_{n=1}^\infty \mathfrak S(C^n).$ Conversely, if $u \in \bigcap_{n=1}^\infty \mathfrak S(C^n)$ then consider $x \in C^\infty$. Then, if $\epsilon >0$ there exists $L \in \bb N$ such that $\e I_k \otimes e + x \in C^L$. By assumption it follows $u \in \mathfrak S(C^L)$ and therefore $u_k(\e I_k \otimes e + x) = \e I_{k^2} + u_k(x) \in M_{k^2}^+$. Since this holds for all $\e >0$ we have $u_k(x) \in M_{k^2}^+$ which implies that $u \in \mathfrak S(C^{\infty})$, since $x \in C^{\infty}$ was arbitrary. So $\bigcap_{n=1}^\infty \mathfrak S(C^n) \subseteq \mathfrak S(C^\infty)$. \qedhere
\end{proof}

\begin{rmk}
    In the above Lemma, and in the following results, we leave open the possibility that the cone $C^{\infty}$ is non-proper. When this is the case, any element $x \in \cc V$ satisfying $\pm x \otimes I_k \in C^{\infty}$ must be a element of the kernel of any $\varphi \in \mathfrak S(C^{\infty})$. We note that even when $C^{\infty}$ is non-proper, it is still the case that $x \in (C^{\infty})_m^{k-\text{min}}$ if and only if $\varphi_m(x) \geq 0$ for each $\varphi \in \mathfrak S(C^{\infty})$. This can be seen by passing to the quotient space $\cc V / \cc J$, where $\cc J := \text{span}\{x : \pm x \otimes I_k \in C^{\infty}\}$ and noting that $y \in \cc J$ if and only if $\varphi(y)=0$ for every $\varphi \in \mathfrak S(C^{\infty})$. Alternatively, one can check that the arguments in the proofs of \cite[Propositions 3.7, 3.8, 3.13]{araiza2021matricial} to see that the result holds when $C$ fails to be proper.
\end{rmk}

Our next result is to prove the coincidence of two order structures. In particular, we consider, for $m \in \mathbb{N}$, the cones
\[ (C^{\infty})_{m}^{k-\text{min}} = ( \lim_{n \to \infty} C^{n} )_{m}^{k-\text{min}} \]
and the $m$-inductive limit $\lim_{n \to \infty} [(C^n)_{m}^{k-\text{min}}].$ To say $x \in  (C^{\infty} )_{m}^{k-\text{min}}$ is to say that for every unital $k$-positive map $\vp: (\cc V, C^\infty, e) \to M_k$, $\vp_m(x) \in M_{km}^+$ by Theorem \ref{thm: k-min characterization}. On the other hand, to say $x \in \lim_{n \to \infty} [(C^n)_{m}^{k-\text{min}}]$ implies for every $\e >0$ there exists $L \in \bb N$ such that $x + \e I_m \otimes e \in (C^L)_m^{k\text{-min}}.$ In other words, we prove that the $k$-inductive limit ``commutes'' with taking the $m$\textsuperscript{th}-cone of the $k$-minimal structure, where $m \in \bb N$. 

\begin{lem}\label{lem: inductive limit commutes with k-min structure}
Let $\cc V$ be a finite-dimensional $*$-vector space and assume $\{C^n\}_{n \in \bb N}$ is a nested increasing sequence of cones such that $(\mathcal{V}, C^n, e)$ is a $k$-AOU space for every $n \in \mathbb{N}$. Then
\[ ( \lim_{n \to \infty} C^{n} )_{m}^{k-\text{min}} = \lim_{n \to \infty} [(C^n)_{m}^{k-\text{min}}]. \]
\end{lem}

\begin{proof}
Let $x \in \lim_{n \to \infty} [(C^n)_m^{k-\text{min}}]$ and let $\vp \in \mathfrak S(C^\infty).$ If $\e >0$ then there exists $L \in \bb N$ such that $\e I_m \otimes e + x \in (C^L)_m^{k-\text{min}}.$ By Lemma~\ref{lem: states on inductive k-limit} it follows that $\vp \in \mathfrak S(C^L)$ and thus $\vp_m(\e I_m \otimes e + x) = \e I_{km} + \vp_m(x) \in M_{km}^+$. Thus $\e I_{km} + \vp_m(x) \in M_{km}^+$ for all $\e >0$. This proves $\vp_m(x) \in M_{km}^+$, and thus $x \in (C^\infty)_m^{k-\text{min}}.$ 

Conversely, suppose $x \notin \lim_{n \to \infty} [(C^n)_m^{k-\text{min}}]$. Then there exists $\e > 0$ such that $\e I_m \otimes e + x \notin (C^L)_m^{k-\text{min}}$ for any $L \in \mathbb{N}$. Hence for each $L \in \mathbb{N}$ there exists $\varphi^L \in \mathfrak S(C^L)$ such that $\e I_{mk} + \varphi^L_m(x) = \varphi^L_m(\e I_m \otimes e + x) \notin M_{mk}^+$. Since $\mathfrak S(C^1)$ is weak*-compact and contains each set $\mathfrak S(C^L)$, there exists a weak* limit point $\varphi$ for the set of linear maps $\{ \varphi^L \}_{L=1}^\infty$ in $\mathfrak S(C^1)$. This limit point must satisfy $\e I_{mk} + \varphi_m(x) \notin M_{mk}^+$ and hence $\varphi_m(x) \notin M_{mk}^+$. However, we claim that $\varphi \in \mathfrak S(C^{\infty})$.  To see this, let $L \in \mathbb{N}$ and let $y \in C^L$.  Since $C^L \subseteq C^{L'}$ for all $L' \geq L$, it follows that $\varphi_k^{L'}(y) \geq 0$ for all $L' \geq L$. Since $\varphi$ is a weak* limit point, $\varphi_k(y) \geq 0$. Hence $\varphi \in \mathfrak S(C^L)$ for every $L \in \mathbb{N}$. By Lemma \ref{lem: states on inductive k-limit} we conclude that $\varphi \in \mathfrak S(C^{\infty})$. Therefore $x \notin (C^\infty)_m^{k-\text{min}}$ since $\varphi_m(x) \notin M_{mk}^+$. \qedhere
\end{proof}

We now define a $k$-AOU space structure on the $*$-vector space $\cc U_{\cc R}$. We will then show that the resulting $k$-AOU space satisfies a universal property for $k$-AOU spaces generated by projections satisfying $\mathcal{R}$.

\begin{defn}\label{defn: k-inductive limit}
Let $(\cc U_{\cc R},E_k,e)$ be the $k$-AOU space where $\cc U_{\cc R}$ is the universal $*$-vector space with generators $\{e,p_1,\dots,p_N\}$ satisfying relations $\cc R$, and $E_k:= E^{max}_k,$ where $E$ is as in Definition~\ref{defn: contraction cone}. Set $E^0 = E_k,$ and define
\[ E^n = E^{n-1}[p_i] \quad \text{ where } \quad i = n \text{ mod } N. \]
We let $E^{\text{proj}(k)} := E^{\infty}$ be the Archimedean closure of $\cup_n E^n$ (i.e. the $k$-inductive limit of $\{E^n\}$).
\end{defn}

\begin{thm}\label{thm: inductive limit is proper}
Let $\cc U_{\cc R}$ be the universal $*$-vector space with generators $\{e,p_1,\dotsc,p_N\}$ satisfying $\cc R$. Then $(\cc U_{\cc R}, E^{\text{proj}(k)}, e)$ is a $k$-AOU space whenever there exists a $k$-AOU space $(\cc V, \tilde{C}, f)$, with projections $\{q_1,\dotsc,q_N\}$ satisfying $\cc R,$ and such that $\cc V = \text{span} \{f, q_1, q_2, \dots, q_n\}$ and $\dim(\cc U_{\cc R}) = \dim(\cc V).$ Moreover, whenever $\mathcal{V}$ is a $k$-AOU space with unit $f$ and projections $\{q_1, q_2 \dots, q_n\}$ satisfying $\mathcal{R}$, then then the map $\pi: \mathcal{U}_{\mathcal{R}} \to \mathcal{V}$ defined by $\pi(e)=f$ and $\pi(p_i) = q_i$ for all $i \in [N]$ is a unital $k$-positive map.
\end{thm}

\begin{proof}
Suppose $(\cc V, \tilde{C}, f)$ is a $k$-AOU space with projections (in the $k$-AOU sense) $\{q_1,\dotsc,q_N\}$ satisfying $\cc R$. Endow this space with the $k$-minimal operator system structure $(\cc V, \tilde{C}^{\text{k-min}}, f)$. By Corollary \ref{cor: C=C[p] means proj in k-min} (see also \cite[Theorem 6.9]{araiza2021matricial}) we have that $\{q_1,\dotsc,q_N\}$ are abstract projections in the operator system $(\cc V, \tilde{C}^{k-min}, f).$ We claim $\pi(E^\infty) \subseteq \tilde{C}$. 

By Theorem~\ref{thm: Universal contraction system}, $\pi_k(E_k) \subseteq \tilde{C}_k^{k-\text{min}} = \tilde{C}$. We proceed by induction (recalling that $E_k = E^0$). Suppose that $\pi_k(E^{n-1}) \subseteq \tilde{C}$ for some $n \in \mathbb{N}$. Let $x \in E^n$. Then $x \in E^{n-1}[p]$ for some $p \in \{p_1, \dots, p_n\}$. So for each $\e>0$ there exists $t>0$ such that \[
(\a + \b)^*x(\a+\b) + [\e \a^*\a + t\b^*\b] \otimes p + [t\a^*\a + \e \b^*\b] \otimes p^\perp \in E^{n-1}
\] for all $\a,\b \in M_k$. We then have \[
(\a+\b)^*\pi_k(x)(\a+\b) + [\e \a^*\a + t\b^*\b] \otimes \pi(p) + [t\a^*\a + \e \b^*\b] \otimes \pi(p^\perp) \in \tilde{C}.
\] Thus, $\pi_k(x) \in \tilde{C}[q],$ where $q = \pi(p).$ Since $q$ is a projection in $(\cc V_{\cc R}, \tilde{C}, f)$, then $\pi_k(x) \in \tilde{C}[q] = \tilde{C},$ proving $\pi_k(E^n) \subseteq \tilde{C}.$ It follows that we have $\pi_k(\bigcup_{L \in \bb N} E^L) \subseteq \tilde{C}.$ If $x \in E^\infty$ then for every $\e>0$ there exists $L \in \bb N$ such that $x + \e I_k \otimes e \in E^L$. This implies $\pi_k(x) + \e I_k \otimes f \in \tilde{C}$ for all $\e>0$, and thus $\pi_k(x) \in \tilde{C}$. This proves $\pi_k(E^\infty) \subseteq \tilde{C}$. So $\pi$ is $k$-positive.

Now suppose also that $\mathcal{V} = \text{span} \{f, q_1, q_2, \dots, q_n\}$ and $\dim(\mathcal{V}) = \dim(\mathcal{U}_{\mathcal{R}})$. A consequence of this is that $E^\infty$ is proper, since $\tilde{C}$ is proper. One sees this by considering the map $\pi: \mathcal{U}_{\mathcal{R}} \to \mathcal{V}$. If $\pm x \in E^{\infty}$, then $\pm \pi(x) \in \tilde{C}$ and hence $\pi(x)=0$. But $\pi$ is injective since it is linear, maps generators to generators, and since $\dim(\mathcal{V}) = \dim(\mathcal{U}_{\mathcal{R}})$. So $x = 0$. Since $E^{\infty}$ is proper, $(\cc U_{\cc R}, E^{\infty}, e)$ is a $k$-AOU space.
\end{proof}

We conclude this section by showing that $p_1, p_2, \dots, p_N$ are abstract projections in the $k$-AOU space $(\cc U_{\cc R}, E^{\text{proj}(k)}, e)$.

\begin{thm}\label{thm: universal quantum k-aou}
Let $(\cc U_{\cc R}, E^{\text{proj}(k)}, e)$ be as in Theorem~\ref{thm: inductive limit is proper}. Then for each $p \in \{p_1,\dotsc,p_{N}\},$ it follows $E^{\text{proj}(k)}[p]= E^{\text{proj}(k)}.$ In other words, the positive contractions $\{p_1,\dotsc,p_N\} \subset \cc U_{\cc R}$ are abstract projections in the $k$-AOU space $(\cc U_{\cc R}, E^{\text{proj}(k)}, e).$
\end{thm}

\begin{proof}
By Theorem~\ref{thm: inductive limit is proper}, the triple $(\cc U_{\cc R}, E^{\text{proj}(k)}, e)$ is a $k$-AOU space. We claim that each $p \in \{p_1,\dotsc,p_N\}$ is an abstract projection in the $k$-AOU sense, i.e. $E^{\text{proj}(k)} = E^{\text{proj}(k)}[p]$.

 Given $p \in \{p_1, \dotsc, p_{N}\} \subset \cc U_{\cc R}$, we need only prove $E^\infty[p] \subseteq E^\infty.$ Let $x \in E^\infty[p]$, which by Lemma~\ref{lem: alternate characterization of C[p]} implies for all $\e>0$ there exists $t>0$ such that \[
\begin{pmatrix} x & x \\ x & x \end{pmatrix} + \epsilon \begin{pmatrix} p & 0 \\ 0 & p^{\perp} \end{pmatrix} \otimes I_k + t \begin{pmatrix} p^{\perp} & 0 \\ 0 & p \end{pmatrix} \otimes I_k \in (E^\infty)_{2k}^{k-\text{min}}.
\] Applying Lemma~\ref{lem: inductive limit commutes with k-min structure} we have for all $\e>0$ there exists $t>0$ such that 
\begin{equation}
\begin{pmatrix} x & x \\ x & x \end{pmatrix} + \epsilon \begin{pmatrix} p & 0 \\ 0 & p^{\perp} \end{pmatrix} \otimes I_k + t \begin{pmatrix} p^{\perp} & 0 \\ 0 & p \end{pmatrix} \otimes I_k \in \lim_{n \to \infty}[(E^n)_{2k}^{k-\text{min}}] %\label{eq: inductive limit containment}
\nonumber
\end{equation} Let $\e \in (0,1)$. Then there exists $\hat t >0$ and $L \in \bb N$ such that \[
\begin{pmatrix} x & x \\ x & x \end{pmatrix} + (\epsilon-\e^2) \begin{pmatrix} p & 0 \\ 0 & p^{\perp} \end{pmatrix} \otimes I_k + \hat t \begin{pmatrix} p^{\perp} & 0 \\ 0 & p \end{pmatrix} \otimes I_k \in (E^L)_{2k}^{k-\text{min}}.
\] Choose $r>0$ such that  $\begin{pmatrix}
\e & 1 \\
1 & r
\end{pmatrix} \in M_2^+$. It then follows that \[
\begin{pmatrix} x & x \\ x & x \end{pmatrix} + (\epsilon-\e^2) \begin{pmatrix} p & 0 \\ 0 & p^{\perp} \end{pmatrix} \otimes I_k + \hat t \begin{pmatrix} p^{\perp} & 0 \\ 0 & p \end{pmatrix} \otimes I_k + \e \begin{pmatrix}
\e & 1 \\
1 & r
\end{pmatrix} \otimes p \otimes I_k + \e \begin{pmatrix}
r & 1 \\
1 & \e
\end{pmatrix} \otimes p^{\perp} \otimes I_k  \in (E^L)_{2k}^{k-\text{min}}.
\] At this point we condense the terms and arrive at \[
\begin{pmatrix}
x + \e I_k \otimes e & x + \e I_k \otimes e \\
x + \e I_k \otimes e & x + \e I_k \otimes e
\end{pmatrix} + s I_k \otimes \begin{pmatrix}
p^\perp & 0 \\
0 & p
\end{pmatrix} \in (E^L)_{2k}^{k-\text{min}},
\] where $s= \e r + \hat{t} -\e.$ This implies for each $\delta >0$ we have \[
\begin{pmatrix}
x + \e I_k \otimes e & x + \e I_k \otimes e \\
x + \e I_k \otimes e & x + \e I_k \otimes e
\end{pmatrix} + \delta I_k \otimes \begin{pmatrix}
p & 0 \\
0 & p^\perp
\end{pmatrix} + s I_k \otimes \begin{pmatrix}
p^\perp & 0 \\
0 & p
\end{pmatrix} \in (E^L)_{2k}^{k-\text{min}},
\] and therefore $x + \e I_k \otimes e \in E^L[p].$ By the definition of the cones $E^n$, it follows $x + \e I_k \otimes e \in E^{\widehat{L}}$ for some $\widehat{L} > L$. In particular, since $p = p_i$ for some $i \in [N]$, and since the family $\{E^n\}_{n \in \bb N}$ forms a nested increasing sequence of cones, we may choose $\widehat{L} \in \bb N$ such that $i = \widehat{L} \mod N$. Then $E^L[p] \subseteq E^{\widehat{L}-1}[p] = E^{\widehat{L}}.$ This proves $x \in E^\infty$. \qedhere
\end{proof}

\section{Dual hierarchy for quantum correlations}
We now wish to apply our results to operator systems whose generators correspond to correlations. We recall some facts concerning correlations. Let $n,m \in \mathbb{N}$. We call a tuple $p = \{p(a,b|x,y) : a,b \in [m], x,y \in [n]\}$ a \textbf{correlation} with $n$ inputs and $m$ outputs if for each $a,b \in [m]$ and $x,y \in [n]$, $p(a,b|x,y)$ is a non-negative real number, and for each $x,y \in [n]$ we have \[ \sum_{a,b = 1}^m p(a,b|x,y) = 1. \] We let $C(n,m)$ denote the set of all correlations with $n$ inputs and $m$ outputs. A correlation $p$ is called \textbf{nonsignalling} if for each $a,b \in [m]$ and $x,y \in [n]$ the values \[ p_A(a|x) := \sum_d p(a,d|x,w) \quad \text{ and } \quad p_B(b|y) := \sum_c p(c,b|z,y) \] are well-defined, meaning that $p_A(a|x)$ is independent of the choice of $w \in [n]$ and $p_B(b|y)$ is independent of the choice of $z \in [n]$. We let $C_{ns}(n,m)$ denote the set of all nonsignalling correlations with $n$ inputs and $m$ outputs.

Much of the literature on correlation sets is focused on various subsets of the nonsignalling correlation sets. We mention three of these subsets here, namely the quantum commuting, quantum, and local correlations. A correlation $p$ is a \textbf{quantum commuting} correlation with $n$ inputs and $m$ outputs if there exists a Hilbert space $H$, a pair of C*-algebras $\mathcal{A}, \mathcal{B} \subseteq B(H)$ with $z_1z_2=z_2z_1$ for all $z_1 \in \mathcal{A}$ and $z_2 \in \mathcal{B}$, projection-valued measures $\{E_{x,a}\}_{a=1}^m \subseteq \mathcal{A}$ and $\{F_{y,b}\}_{b=1}^m \subseteq \mathcal{B}$ for each $x,y \in [n]$, and a state $\phi: \mathcal{A} \mathcal{B} \to \mathbb{C}$ such that $p(a,b|x,y) = \phi(E_{x,a} F_{y,b})$ for all $a,b \in [m]$ and $x,y \in [n]$. A quantum commuting correlation is called a \textbf{quantum} correlation if we require the Hilbert space $H$ to be finite-dimensional. A quantum commuting correlation is called \textbf{local} if we require that the C*-algebras $\mathcal{A}$ and $\mathcal{B}$ are commutative. We let $C_{qc}(n,m), C_{q}(n,m)$, and $C_{loc}(n,m)$ denote the sets of quantum commuting, quantum, and local correlations, respectively. We let $C_{qa}(n,m) := \overline{C_q(n,m)}$, and say such correlations are \textbf{quantum approximate}.

It is well-known that for each input-output pair $(n,m)$ each of the correlation sets mentioned above are convex subsets of $\mathbb{R}^{n^2 m^2}$ and satisfy
\[ C_{loc}(n,m) \subseteq C_{q}(n,m) \subseteq C_{qa}(n,m) \subseteq C_{qc}(n,m) \subseteq C_{ns}(n,m) \subseteq C(n,m). \]
Moreover, each inclusion in the above sequence is proper for some choice of input $n$ and output $m$. The fact that local correlations are a proper subset of quantum correlations goes back to John Bell \cite{bell1964einstein}. The proper inclusion of the quantum correlations inside the quantum approximate correlations was due to Slofstra \cite{slofstra2019set}, and that quantum approximate correlations are a proper subset of the quantum commuting correlations is due to Ji-Natarajan-Vidick-Wright-Yuen \cite{ji2020mip}.

We recall the following definitions from \cite{araiza2020abstract} and \cite{araiza2021matricial}.

\begin{defn} \label{defn: qc operator system}
Let $n, m \in \N$. We call a pair $(\mathcal{V}, \{Q(a,b|x,y)\}_{a,b \in [m], x,y \in [n]})$ a \textbf{nonsignalling} vector space on $n$ inputs and $m$ outputs if $\mathcal V$ is a vector space spanned by vectors $\{Q(a,b|x,y): a,b \in [m], x,y \in [n] \}$ satisfying \[ \sum_{a,b=1}^m Q(a,b|x,y) = e \] for some fixed nonzero vector $e$, which we call the \textbf{unit} of $\mathcal{V}$, and such that the vectors \[ E(a|x) := \sum_{c=1}^m Q(a,c|x,z) \quad \text{ and } \quad F(b|y) := \sum_{d=1}^m Q(d,b|w,y) \] are well-defined. When the vectors $Q(a,b|x,y)$ are clear from context, we simply call $\mathcal{V}$ a nonsignalling vector space. When $\mathcal{V}$ is nonsignalling, we write $n(\mathcal{V})$ and $m(\mathcal{V})$ for the number of inputs and for the number of outputs, respectively; i.e., $\mathcal{V} = \Span \{Q(a,b|x,y) : a,b \in [m(\mathcal{V})], x,y \in [n(\mathcal{V})] \}.$ 

A \textbf{nonsignalling operator system} is an operator system structure $(\mathcal{V}, \mathcal{C}, e)$ on a non-signalling vector space $\mathcal{V} = \Span \{Q(a,b|x,y)\}_{a,b \in [m]; x,y \in [n]}$ where $Q(a,b|x,y) \in \mathcal C_1$ for each $a,b \in [m]$ and $x,y \in [n]$. We call $\mathcal{V}$ a \textbf{quantum commuting operator system} if it is a nonsignalling operator system with the additional condition that each $Q(a,b|x,y)$ is an abstract projection in $(\mathcal{V}, \mathcal{C}, e)$.

A \textbf{quantum $k$-AOU space} is a $k$-AOU space structure $(\cc V, C, e)$ on a nonsignalling vector space $\cc V = \Span\{Q(a,b|x,y): a,b \in [m], x,y \in [n]\}$ with the additional condition that each $Q(a,b|x,y)$ is an abstract projection in the $k$-AOU sense.
\end{defn}

If $(\cc V, \cc C, e)$ is a quantum commuting operator system which is $k$-minimal, then $(\cc V , \cc C_k, e)$ is a quantum $k$-AOU space (by Corollary \ref{cor: C=C[p] means proj in k-min}). Therefore for any $k \in \mathbb{N}$ we call a $k$-minimal quantum commuting operator system a \textbf{quantum operator system}.

\begin{thm}[{\cite[Theorem 6.3]{araiza2020abstract}}, {\cite[Theorem 7.4]{araiza2021matricial}}] \label{thm: NS and QC operator systems}
A correlation $p \in C(n,m)$ is nonsignalling (resp. quantum commuting, quantum) if and only if there exists a nonsignalling (resp. quantum commuting, quantum) operator system $\mathcal{V}$ with generators $\{Q(a,b|x,y); a,b \in [m], x,y \in [n] \}$ and a state $\phi$ on $\mathcal{V}$ such that 
\[ p(a,b|x,y) = \phi(Q(a,b|x,y)) \]
for each $a,b \in [m]$ and $x,y \in [n]$.
\end{thm}

Section~\ref{sec: univeral projection system} and Section~\ref{sec: universal k-relation AOU} pertained to constructing universal operator systems spanned by projections satisfying relations $\cc R,$ and its $k$-AOU space counterpart, respectively. In order to guarantee properness of the inductive limit matricial ordering, or $k$-inductive limit as in Section~\ref{sec: universal k-relation AOU}, we assumed that there existed a similar object, satisfying the relations in question, and such that the corresponding $*$-vector spaces had equal dimension. We recall Example 5.5 from \cite{araiza2021universal}.

\begin{ex} \label{ex: nonsignalling ex}
Let $n,m \in \mathbb{N}$. Consider the $*$-vector space \begin{align*}
\cc V:= \span \{Q(a,b|x,y)\}_{x,y \in [n], a,b \in [m]}, \quad Q(a,b|x,y):= I_m^{\otimes x -1} \otimes E_a \otimes I_m^{\otimes n-x} \otimes I_m^{\otimes y-1} \otimes E_b \otimes I_m^{\otimes n-y} \in D_m^{\otimes n^2}, 
\end{align*} where $E_a \in M_m$ denotes the diagonal $m \times m$ matrix with a 1 in the $a$\textsuperscript{th} entry and zeroes elsewhere, $I_m^{\otimes n}$ denotes the n-fold tensor product of the $m \times m$ identity matrix with itself, with the understanding $I_m^o = 1.$ Then $\cc V$ is a nonsignalling vector space. As shown in {\cite[Example 5.5]{araiza2021universal}}, $\dim(\cc V) = (n(m-1)+1)^2$.
\end{ex}

Let $n,m \in \mathbb{N}$. For the remainder of this section, let $\mathcal{R}$ denote the nonsignalling relations on the vectors $e$ and $\{P(a,b|x,y)\}_{a,b \in [m], x,y \in [n]}$. Specifically, $\mathcal{R}$ includes the following relations: for each $x,y$, let $r_1^{x,y}$ denote the relation
\[ \sum_{a,b=1}^m P(a,b|x,y) = e, \]
for each $a \in [m]$ and $x,y,z \in [n]$, let $r_2^{a,x,y,z}$ and $r_3^{a,x,y,z}$ denote the relations
\[ \sum_{b=1}^m P(a,b|x,y) = \sum_{b=1}^m P(a,b|x,z) \quad \text{and} \quad \sum_{b=1}^m P(b,a|y,x) = \sum_{b=1}^m P(b,a|z,x) \]
respectively.

Let $\cc V_{ns} := \cc U_{\cc R}$. We leave it to the reader to verify that $\dim(\cc V_{ns}) = (n(m-1)+1)^2$. In light of Example \ref{ex: nonsignalling ex} and the results of Sections 3, 4, and 5, we have the following results.

\begin{cor} \label{cor: ns,q,qc universal systems}
Suppose that $\mathcal{V} = \{Q(a,b|x,y)\}_{a,b \in [m], x,y \in [n]}$ is a non-signalling vector space. Then the map $\pi: \cc V_{ns} \to \cc V$ defined by $\pi(P(a,b|x,y)) = Q(a,b|x,y)$ is a well-defined linear map. Moreover:
\begin{enumerate}
    \item If $(\cc V, \mathcal{C}, e)$ is a nonsignalling operator system, then $\pi$ is completely positive on the nonsignalling operator system $(\cc V_{ns}, E^{\text{max}}, e)$.
    \item If $(\cc V, \mathcal{C}, e)$ is a quantum commuting operator system, then $\pi$ is completely positive on the quantum commuting operator system $(\cc V_{ns}, \mathcal{E}^{\text{proj}}, e)$.
    \item If $(\cc V, \cc C, e)$ is a $k$-minimal quantum operator system for some $k \in \mathbb{N}$, then $\pi$ is completely positive on the ($k$-minimal) quantum operator system $(\mathcal{V}_{ns}, (E^{\text{proj}(k)})^{k-\text{min}}, e)$.
\end{enumerate}
\end{cor}

Analogous with results in \cite{araiza2020abstract,araiza2021universal,araiza2021matricial}, we provide a characterization of various nonsignalling correlations as images of unital linear functionals on the universal nonsignalling vector space. In the following theorem, we let $D_{ns} := E$ and $D_{qc} := \mathcal{E}^{\text{proj}}_1$. Let $k \in \mathbb N$. As proven in \cite[Corollary 3.18]{araiza2021matricial}, given any extension $\mathcal E$ of $E^{\text{proj}(k)}$, the first $k$ cones $\mathcal E_1,\dotsc, \mathcal E_{k}$, are uniquely determined by the cone $E^{\text{proj}(k)} \subseteq M_k(\cc V_{ns})$. In particular, any two extensions yield the same order structure up to level $k$. Thus, we unambiguously define the cone \begin{align} \nonumber
D_{q(k)}:=(E^{\text{proj}(k)})_1.
\end{align}
Finally, we define $D_{qa} = \cap_{k=1}^\infty D_{q(k)}$.

\begin{thm} \label{thm: hierarchy AOU spaces}
Let $\mathcal{V}_{ns}$ be the universal nonsignalling vector space with generators $\{P_{ns}(a,b|x,y): x,y \in [n], a,b \in [m]\}$, and let $\varphi: \mathcal{V}_{ns} \to \mathbb{C}$ be a unital linear functional. Let $p(a,b|x,y) := \varphi(P_{ns}(a,b|x,y))$ for all $a,b \in [m]$ and $x,y \in [n]$. Then the following statements are true.
\begin{enumerate}
    \item $p \in C_q(n,m)$ if and only if $\varphi({D}_{q(k)}) \geq 0$ for some $k \in \mathbb{N}$.
    \item $p \in C_{qa}(n,m)$ if and only if $\varphi(D_{qa}) \geq 0$.
    \item $p \in C_{qc}(n,m)$ if and only if $\varphi({D}_{qc}) \geq 0$.
    \item $p \in C_{ns}(n,m)$ if and only if $\varphi({D}_{ns}) \geq 0$.
\end{enumerate}
\end{thm}

\begin{proof}
We begin with Items (3) and (4). If $\vp(D_{ns}) \geq 0$, then $\vp: (\cc V_{ns}, D_{ns}, e_{ns}) \to \bb C$ is a state, and therefore by Corollary \ref{cor: ns,q,qc universal systems} it must follow $p \in C_{ns}(n,m)$ since $(\cc V_{ns}, E^{\text{max}}, e)$ is a nonsignalling operator system. Conversely, if $p \in C_{ns}(n,m)$ then $\vp(P(ab|xy)) \geq 0$ for each generator $P(ab|xy),$ thus by linearity and the definition of $D_{ns}$, it must follow $\vp(D_{ns}) \geq 0.$ This proves Item (4). 

To prove Item (3), suppose $\vp(D_{qc}) \geq 0$. Then since $\vp: (\cc V_{ns}, D_{qc}, e_{ns}) \to \bb C$ is a state, by Corollary \ref{cor: ns,q,qc universal systems} it follows $p \in C_{qc}(n,m)$ since $(\cc V_{ns}, \cc E^{proj}, e)$ is a quantum commuting operator system. Conversely, if $p \in C_{qc}(n,m)$, then by Theorem \ref{thm: NS and QC operator systems} there exists a quantum commuting operator system $\cc V$ with generators $\{Q(a,b|x,y)\}$ and a state $\psi: \cc V \to \bb C$ such that $\psi(Q(a,b|x,y)) = p(a,b|x,y).$ By Corollary \ref{cor: ns,q,qc universal systems}, 
\[ p(a,b|x,y) = \psi(Q(a,b|x,y)) = \psi(\pi(P(a,b|x,y)) \]
where $\pi: \cc V_{ns} \to \cc V$ is the universal mapping. It follows that $\psi \circ \pi = \vp$, proving Item (3).

We now prove Item (1). Suppose $\vp(D_{q(k)}) \geq 0.$ By Theorem~\ref{thm: universal quantum k-aou}, $(\cc V_{ns}, E^{\text{proj}(k)}, e)$ is a quantum $k$-AOU space, and thus has a (quantum) operator system structure with a $k$-minimal matrix ordering $\cc C$ such that $\cc C_1 = D_{q(k)}$. Since $\vp$ is a state on this operator system, $p \in C_q(n,m)$ by Theorem~\ref{thm: NS and QC operator systems}. Conversely, suppose $p \in C_q(n,m)$. Then by Theorem \ref{thm: NS and QC operator systems} there exists a ($k$-minimal) quantum operator system $(\cc V, \cc C, e) = \Span\{Q(ab|xy)\}$  and a state $\psi: \cc V \to \bb C$ such that $\psi(Q(ab|xy)) = p(ab|xy)$. By item (3) of Corollary \ref{cor: ns,q,qc universal systems},
\[ p(a,b|x,y) = \psi(Q(a,b|x,y)) = \psi(\pi(P(a,b|x,y)) \]
where $\pi: \cc V_{ns} \to \cc V$ is the universal mapping. Again, $\psi \circ \pi = \vp$, proving item (1).

We now prove Item (2). Let $p \in C_{qa}(n,m)$, which we write as $p = \lim_i p^i$, where each $p^i \in C_q(n,m)$. By Theorem~\ref{thm: NS and QC operator systems}, for each $i \in I$, there exists a quantum $k(i)$-AOU space, $\cc V^i$, and a state $\vp^i: \cc V^i \to \bb C$ such that $p^i(a,b|x,y) = \vp^i(a,b|x,y)$ for all $x,y \in [n], a,b \in [m].$ For each $i \in I$, let $\pi^i: \cc V_{ns} \to \cc V^i$ denote the universal mapping obtained via Theorem~\ref{thm: inductive limit is proper}. Then it follows $\vp = \lim_i \vp^i \pi^i.$ If $x \in D_{qa}$, then $x \in D_{qk(i)}$ for each $i,$ and consequently $\vp^i\pi^i(x) \geq 0$ for every $i$. Therefore, it must follow $\vp(x) \geq 0$. 
 
Conversely, suppose $\vp(D_{qa}) \geq 0$. We will show that there exists a sequence of states $\vp^k: (\cc V_{ns}, D_{q(n_k)}, e_{ns}) \to \bb C$ which converge in norm to $\vp$ with respect to the norm on the dual of $(\cc V_{ns}, D_{qa}, e_{ns})$. Since $\cc V_{ns}$ is finite dimensional, all norms on the dual of $\cc V_{ns}$ are equivalent, and convergence in norm will imply $\sigma(\cc V_{ns}^*, \cc V_{ns})$ convergence, so that the quantum correlations defined by $p^k(a,b|x,y):= \vp^k(P_{ns}(a,b|x,y))$ converge to the quantum approximate correlation defined by $p(a,b|x,y) = \vp(P_{ns}(a,b|x,y))$.

Suppose that $\vp'(D_{qa}) \geq 0$ and that $\vp'$ is an interior point of the set of positive states, meaning that for some $\epsilon > 0$ we have that if $\psi$ is hermitian and $\|\psi - \vp'\| < \epsilon$ then $\psi(D_{qa}) \geq 0$. We claim that if $x \in D_{qa}$ and $\vp'(x)=0$ then $x=0$. Indeed, suppose that $\vp'(x)=0$. If $\psi(x) > 0$ for some other state on $D_{qa}$, then for every $t > 0$, $\vp'(x) - t(\psi(x) - \vp'(x)) < 0$. Setting $\rho_t = \vp' - t(\psi - \vp')$, we may choose $t > 0$ small enough so that $\|\vp' - \rho_t\| < \epsilon$ and hence $\rho_t$ is a state. However this is a contradiction since $\rho_t(x) < 0$. It follows that $\vp'$ is strictly positive on all non-zero elements of $D_{qa}$. 

Next, we claim that $\vp'$ is positive on $D_{q(k)}$ for some $k \in \mathbb{N}$. Suppose this is not the case. Then for every $k \in \mathbb{N}$, there exists $x_k \in D_{q(k)}$ such that $\vp'(x_k) < 0$. We may assume that $\|x_k\| = 1$ for all $k$, where the norm is calculated in the AOU space $(\cc V_{ns}, D_{qa}, e_{ns})$. Since $V_{ns}$ is finite dimensional, the set of norm one hermitian elements is closed and bounded and therefore compact. Let $x'$ be a limit point of $\{x_k\}$. Then $x' \in D_{qa}$, $\|x'\|=1$ and $\vp'(x) \leq 0$. This is a contradiction, since $\vp'$ is strictly positive on non-zero elements of $D_{qa}$. This proves the claim. Since $\vp$ is the norm limit of a sequence $\{\vp^k\}$ of interior points of the state space of $(\cc V_{ns}, D_{qa}, e_{ns})$, we have proven Item (2). \qedhere
\end{proof}

\begin{rmk}
In the notation of Theorem \ref{thm: hierarchy AOU spaces}, it is also true that $p \in C_{loc}(n,m)$ if and only if $\vp(D_{q(1)}) \geq 0$. We leave the proof to the interested reader. The main point is that $D_{q(1)}$ defines a $1$-minimal quantum commuting operator system $(\cc V_{ns}, \cc C, e)$ and every $1$-minimal quantum commuting operator system can be realized as a subsystem of a commutative C*-algebra (e.g. see the construction of OMIN$(V)$ in \cite{PaulsenTodorovTomfordeOpSysStructures}). In fact, $(\cc V_{ns}, \cc C, e)$ is completely order isomorphic to the operator system presented in Example \ref{ex: nonsignalling ex}.
\end{rmk}

\begin{rmk}
Theorem \ref{thm: hierarchy AOU spaces} defines a hierarchy of cones
\[ D_{ns} \subseteq D_{qc} \subseteq D_{qa} \subseteq \dots \subseteq D_{q(3)} \subseteq D_{q(2)} \subseteq D_{q(1)} \]
each making $\cc V_{ns}$ into an AOU space. This hierarchy is dual to the hierarchy of correlation sets in the sense of Kadison duality [REF]: to every closed compact convex subset $C$ of a complex topological vector space, there corresponds an AOU space, namely the affine functions on $C$; and conversely to every AOU space $V$, there corresponds a compact convex subset of its dual space, namely the state space of $V$. Under this duality, we may realize $(\cc V_{ns}, D_{ns}, e)$ as the affine functions on $C_{ns}(n,m)$, $(\cc V_{qc}, D_{qc}, e)$ as the affine functions on $C_{qc}(n,m)$, and $(\cc V_{ns}, D_{qa}, e)$ as the affine functions on $C_{qa}(n,m)$. It follows from \cite{junge2011connes} and \cite{OzawaAboutConnes} that Connes' embedding problem is equivalent to asking if $D_{qc} = D_{qa}$ for all parameters $n,m \in \mathbb{N}$.
\end{rmk}

\section{SIC-POVMs}

We conclude with an application to the existence question for symmetric informationally complete positive operator-valued measures (SIC-POVMs). We find a new necessary condition for the existence of a SIC-POVM based on the universal operator system constructions above. We make use of the following well-known fact [Ref].

\begin{lem} \label{lem: Hilbert-Schmidt}
Let $d \in \mathbb{N}$ and let $\tau(x) = \frac{1}{d}\Tr(x)$ denote the normalized trace of $x \in M_d$. Then $M_d$ is a Hilbert space with respect to the Hilbert-Schmidt inner product $\langle x,y \rangle_{HS} := \tau(x^*y)$. If $x,y \geq 0$ then $\langle x,y \rangle \geq 0$. Furthermore, if $\langle x,y \rangle \geq 0$ for all $y \geq 0$, then $x \geq 0$.
\end{lem}

\begin{defn}
Let $\mathcal{V}$ be a finite dimensional operator system, and let $\langle \cdot, \cdot \rangle: \mathcal{V} \times \mathcal{V} \to \mathbb{C}$ be an inner product making $\mathcal{V}$ into a Hilbert space. Define an inner product $\langle \cdot, \cdot \rangle_n$ on $M_n(\mathcal{V}) \cong M_n \otimes \mathcal{V}$ by $\langle A \otimes x, B \otimes y \rangle_n := \langle A,B \rangle_{HS} \langle x,y \rangle$ for all $A,B \in M_n$ and $x,y \in \mathcal{V}$ and extended to $M_n(\mathcal{V})$ by sesquilinearity of the inner product. We say that $\langle \cdot, \cdot \rangle$ is \textbf{completely positive} if $\langle x,y \rangle_n \geq 0$ whenever $x,y \in M_n(\mathcal{V})^+$. It is called \textbf{completely self-dual} if $\langle x,y \rangle_n \geq 0$ for all $y \geq 0$ implies that $x \geq 0$.
\end{defn}

It follows from Lemma \ref{lem: Hilbert-Schmidt} that the inner product $\langle \cdot, \cdot \rangle_{HS}$ on $M_d$ is completely positive and completely self-dual for all $d \in \mathbb{N}$.

Let $d \in \mathbb{N}$. A \textbf{symmetric informationally complete positive operator-valued measure} in $M_d$ is a set $\{ P_1, \dots, P_{d^2} \} \subseteq M_d$ of projections which satisfy the relations
\[ \sum_{i=1}^{d^2} P_i = d I \]
and
\[ \text{Tr}(P_i P_j) =  \begin{cases} \frac{1}{d+1} & i \neq j \\ 1 & i = j \end{cases} \]
where $\text{Tr}(\cdot)$ denotes the trace function on $M_d$ and $I \in M_d$ denotes the identity matrix. It is conjectured that a SIC-POVM exists in $M_d$ for all $d \in \mathbb{N}$. However, this conjecture has only been verified for some values of $d$. It is an open question whether or not there exists an upper bound on the set of integers $d \in \mathbb{N}$ for which a SIC-POVM exists in $M_d$ \cite{Fuchs2017SICs}.

Suppose there exists a SIC-POVM $\{P_1, \dots, P_{d^2} \}$ in $M_d$. Then the $d$-minimal\footnote{To see $M_d$ is $d$-minimal, consider e.g. Theorem 3.7 of \cite{xhabli2012super}.} operator system $M_d$ is spanned by the projections $P_1, \dots, P_{d^2}$ which satisfy the relation
\[ \sum_{i=1}^{d^2} P_i = d I \]
(that $\{P_1, \dots, P_{d^2}\}$ spans $M_d$ follows from the observation that the matrix $M_{i,j} = \Tr(P_i P_j)$ has rank $d^2$). Let $\mathcal{R}$ denote the relation
\[ \sum_{i=1}^{d^2} p_i = d e \]
and let $\cc U_{\cc R}^d$ denote the vector space $\cc U_{\cc R} = \text{span} \{p_1, \dots, p_{d^2}\}$ equipped with its universal $d$-AOU space structure inherited from the cone $E^{\text{proj}(k)}$ (regarded as a $d$-minimal operator system). By Theorem \ref{thm: inductive limit is proper}, the corresponding universal $d$-minimal matrix-ordered vector space $\mathcal{U}_{\mathcal{R}}^d$ is an operator system (i.e. its matrix-ordering is proper), since we have assumed there exists a SIC-POVM in $M_d$. By Theorem \ref{thm: inductive limit is proper}, the linear map $\pi: p_i \mapsto P_i$ for $i \in [d^2]$ is a unital completely positive map.

It turns out that $\mathcal{U}_{\mathcal{R}}^d$ is an operator system for every $d \in \mathbb{N}$, whether or not a SIC-POVM exists in $M_d$. To show that, we need the following Lemma.

\begin{lem} \label{lem: SIC relation existence}
Let $d \in \mathbb{N}$. Then there exists a $d$-minimal operator system $\mathcal{V}$ of dimension $d^2$ spanned by projections $q_1, \dots, q_{d^2} \in \mathcal{V}$ which satisfy the relation
\[ \sum_{i=1}^{d^2} q_i = d e \]
where $e \in \mathcal{V}$ denotes the identity.
\end{lem}

\begin{proof}
Let $\{e_i\}$ denote the standard basis for $\mathbb{C}^d$ and $E_i = e_i^*e_i$ the corresponding rank one projection. Let $P_1 := I_d \oplus E_1 \in M_{2d}$. For $i=2,3,\dots,d$, let $P_i = 0_d \oplus E_i \in M_{2d}$. For $j = d+1, \dots, d^2$, let $P_j = P_{j \mod d}$ (so that $P_1 = P_{d+1}$, etc). Then $\sum_{i=1}^{d^2} P_i= dI$. Since $P_1, \dots, P_d$ are linearly independent, $\dim \text{span} \{P_1, \dots, P_{d^2} \} = d$.

For each $i \in [d^2]$, define
\[ q_i = \bigoplus_{\sigma \in S_{d^2}} P_{\sigma(i)} \]
where $S_{d^2}$ denotes the permutation group for $[d^2]$. Since 
\[ \sum_{i=1}^{d^2} P_{\sigma(i)} = \sum_{i=1}^{d^2} P_i = dI \]
for every $\sigma \in S_{d^2}$, we see that $\sum q_i = dI$, where $I$ is the identity of $M_{2d|S_{d^2}|}$. 

We claim that $\{q_1, \dots, q_{d^2}\}$ is linearly independent. To see this, suppose that $\sum t_i q_i = 0$. Then $\sum t_i P_{\sigma(i)} = 0$ for every permutation $\sigma \in S_{d^2}$. By the definition of $P_1,\dots,P_{d^2}$,  there exist disjoint subsets $\Gamma_1, \dots, \Gamma_d$ of $[d^2]$ with $|\Gamma_i| = d$ for each $i$ such that $P_{\sigma(i)} = P_k$ whenever $i\in \Gamma_{k}$. Hence
\[ \sum_{i=1}^{d^2} t_i P_{\sigma(i)} = \sum_{k=1}^d ( \sum_{j \in \Gamma_k} t_j) P_k = 0. \]
Since $\{P_1, \dots, P_d\}$ are linearly independent, 
\[ \sum_{j \in \Gamma_k} t_j = 0 \]
for all $k=1,2,\dots,d$. Since this holds for every permutation $\sigma \in S_{d^2}$, we conclude that
\[ \sum_{j \in \Gamma} t_j = 0 \]
for every subset $\Gamma \subseteq [d^2]$ with $|\Gamma| = d$. It is now easily verified that $t_i = 0$ for every $i=1,2,\dots,d^2$.
\end{proof}

\begin{rmk}
In the proof of the preceding lemma, the projections $\{q_1, \dots, q_{d^2}\}$ constructed are mutually commuting. Thus the operator system they span is 1-minimal and hence $d$-minimal for all $d \in \mathbb{N}$. When there exists a SIC-POVM in $M_d$, then the statement can be satisfied for $\mathcal{V} = M_d$. We do not know if the statement can be satisfied for $\mathcal{V} = M_d$ in general.
\end{rmk}

\begin{cor}
Let $\mathcal{R}$ denote the relation
\[ \sum_{i=1}^{d^2} p_i = d e \]
and let $\mathcal{U}_{\mathcal{R}}^d = \text{span} \{p_1, \dots, p_{d^2}\}$ denote the corresponding universal $d$-minimal operator system generated by projections $p_1, \dots, p_{d^2}$ with unit $e$ satisfying $\mathcal{R}$. Then $\mathcal{U}_{\mathcal{R}}^d$ is an operator system (i.e. its matrix ordering is proper) and $\dim(\mathcal{U}_{\mathcal{R}}) = d^2$.
\end{cor}

\begin{proof}
This follows from Theorem \ref{thm: inductive limit is proper} and from Lemma \ref{lem: SIC relation existence}, since $\mathcal{U}_{\mathcal{R}}$ has a proper matrix ordering provided there exists at least one $d$-minimal operator system spanned by projections $Q_1, \dots, Q_{d^2}$ satisfying $\mathcal{R}$ and having full dimension.
\end{proof}

Again, suppose there exists a SIC-POVM $\{P_1,\dots, P_{d^2}\}$ in $M_d$, and let $\mathcal{U}_{\mathcal{R}}^d = \text{span}\{p_1, \dots, p_{d^2}\}$ denote the corresponding universal $d$-minimal operator system. Then the map $\pi: \mathcal{U}_{\mathcal{R}}^d \to M_d$ defined by $\pi(p_i) = P_i$ is unital completely positive by Theorem \ref{thm: inductive limit is proper}. Since $\dim(\mathcal{U}_{\mathcal{R}}^d) = d^2$, $\pi$ is also injective. Therefore we can define an inner product $\langle \cdot, \cdot \rangle$ on $\mathcal{U}_{\mathcal{R}}^d$ by
\[ \langle x,y \rangle := \frac{1}{d} \Tr(\pi(x)^* \pi(y)). \]
Since $\pi(p_i) = P_i$ and $\{P_i\}$ is a SIC-POVM, we can calculate this inner product directly using 
\[ \langle p_i, p_j \rangle =  \begin{cases} \frac{1}{d(d+1)} & i \neq j \\ \frac{1}{d} & i = j \end{cases} \]
This yields new necessary conditions for the existence of a SIC-POVM in terms of the inner product $\langle \cdot, \cdot \rangle$ on $\mathcal{U}_{\mathcal{R}}^d$.

\begin{thm}
Let $d \in \mathbb{N}$. Suppose that there exists a SIC-POVM in $M_d$. Let $\mathcal{V} = \mathcal{U}_{\mathcal{R}}^d = \text{span}\{p_1, \dots, p_{d^2}\}$ denote the universal $d$-minimal operator system spanned by projections $p_1, \dots, p_{d^2}$ satisfying the relation
\[ \sum_{i=1}^{d^2} p_i = de \]
where $e$ denotes the identity. Then the inner product $\langle \cdot, \cdot \rangle$ on $\mathcal{V}$ given by 
\[ \langle p_i, p_j \rangle =  \begin{cases} \frac{1}{d(d+1)} & i \neq j \\ \frac{1}{d} & i = j \end{cases} \]
is completely positive. Moreover, if $\langle \cdot, \cdot \rangle$ is completely self-dual, then there exists a complete order isomorphism $\pi: \mathcal{V} \to M_d$ such that $\{ \pi(p_i)\}$ is a SIC-POVM.
\end{thm}

\begin{proof}
Let $\pi: p_i \mapsto P_i$ be the completely positive whose range is the matrix algebra $M_d$ and where $\{P_1,\dots,P_{d^2}\}$ is a SIC-POVM.
Let $n \in \mathbb{N}$. If $a,b \in M_n(\mathcal{V})_+$, then $\pi_n(a), \pi_n(b) \in M_n \otimes M_d$ are necessarily positive. Hence
\begin{equation} \label{innerProductEqn}
\langle a, b \rangle = \langle \pi_n(a), \pi_n(b) \rangle_{HS} \geq 0
\end{equation}
since $\langle \cdot, \cdot \rangle_{HS}$ is completely positive (here we are using the fact that $\langle p_i, p_j \rangle = \langle P_i, P_j \rangle_{HS}$). We conclude that $\langle \cdot, \cdot \rangle$ is completely positive.

Now assume that $\langle \cdot, \cdot \rangle$ is completely self-dual. Let $\mathcal{C}$ denote the matrix ordering on $\mathcal{V}$. Since $\pi$ is injective, $\mathcal{D} := \{\pi_n(\mathcal{C}_n)\}$ is a matrix ordering on $M_d$. Suppose that $\langle \pi_n(x), \pi_n(a) \rangle_{HS} \geq 0$ for all $\pi_n(a) \in \mathcal{D}_n$. Then $\langle x,a \rangle_n \geq 0$ for all $a \in \mathcal{C}_n$ by Equation \ref{innerProductEqn}. Since $\langle \cdot, \cdot \rangle$ is completely self-dual, $x \geq 0$. Since $\pi$ is completely positive, $\pi_n(x) \geq 0$. It follows that $\langle \cdot, \cdot \rangle_{HS}$ is completely self-dual with respect to the matrix ordering $\mathcal{D}$. Finally, if $X \in (M_n \otimes M_d)^+$ then $\langle A, X \rangle_{HS} \geq 0$ for all $A \in \mathcal{D}_n$. It follows that $X \in \mathcal{D}_n$. So $\mathcal{D}_n = (M_n \otimes M_d)^+$. We conclude that $\pi$ is a complete order isomorphism on $M_d$.
\end{proof}

Since $\langle \cdot, \cdot \rangle$ must be completely positive whenever a SIC-POVM exists, and since $\mathcal{U}_{\mathcal{R}}^d$ is an operator system, we have obtained a new necessary condition for the existence of a SIC-POVM. The following is the obvious logical implication.

\begin{cor}
Let $d \in \mathbb{N}$ and let $\mathcal{V} = \mathcal{U}_{\mathcal{R}}^d = \text{span}\{p_1, \dots, p_{d^2}\}$ denote the universal $d$-minimal operator system spanned by projections $p_1, \dots, p_{d^2}$ satisfying the relation
\[ \sum_{i=1}^{d^2} p_i = de \]
where $e$ denotes the identity. If the inner product $\langle \cdot, \cdot \rangle$ on $\mathcal{V}$ given by 
\[ \langle p_i, p_j \rangle =  \begin{cases} \frac{1}{d(d+1)} & i \neq j \\ \frac{1}{d} & i = j \end{cases} \]
is not completely positive, then there does not exist a SIC-POVM in $M_d$.
\end{cor}

We do not know if $\cc U_{\cc R}^d \cong M_d$ when a SIC-POVM exists. However, if the inner product on $\cc U_{\cc R}^d$ is self-dual then either $\cc U_{\cc R} \cong M_d$ or a SIC-POVM does not exist in $M_d$, by the Corollary above. 

\begin{rmk}
Let $d \in \mathbb{N}$. Two orthonormal bases $\{\phi_1, \dots, \phi_d\}$ and $\{\psi_1, \dots, \psi_d\}$ for $\mathbb{C}^d$ are called \textbf{mutually unbiased} if $|\langle \phi_i, \psi_j \rangle| = \frac{1}{\sqrt{d}}$ for all $i,j \in [d]$. When $d = p^r$ where $p,r \in \mathbb{N}$ and $p$ is prime, it is known that there exist $d+1$ mutually unbiased bases. Equivalently, there exist $d+1$ projection valued measures $\{P_{x,a}\}_{a=1}^d$, $x \in [d+1]$, satisfying
\begin{equation} \label{eqn: MUB}
\Tr(P_{x,a} P_{y,b}) = \begin{cases} \frac{1}{d} & x \neq y \\
1 & x=y, a=b \\
0 & x=y, a \neq b
\end{cases} \end{equation}
When this holds, the projections $\{P_{x,a}\}$ span the matrix algebra $M_d$. When $d$ is not of the form $d=p^r$, it is an open question whether or not there exist $d+1$ mutually unbiased bases. In particular, it is conjectured that there exist no more than three mutually unbiased bases in $\mathbb{C}^6$ (although the nonexistence of seven mutually unbiased bases in $\mathbb{C}^6$ remains an open question) \cite{Raynal2011MUBs}.

Using the methods outlined above for SIC-POVMs, one could consider $\mathcal{U}_{\mathcal{R}}^d$, the universal $d$-minimal operator system spanned by projections $\{p_{x,a} : x=1,\dots,d+1; a=1,\dots,d\}$ satisfying the relations
\[ \sum_{a=1}^d p_{x,a} = e \]
for each $x \in [d+1]$. Whenever there exist $d+1$ mutually unbiased bases exist in $\mathbb{C}^d$, we would obtain a unital completely positive map $\pi: \mathcal{U}_{\mathcal{R}}^d \to M_d$ with $\{\pi(p_{x,a})\}$ corresponding to a family of mutually unbiased bases. Using the inner product described in Equation \ref{eqn: MUB}, we would obtain new necessary conditions on this open existence problem. Since it is believed that $d+1$ mutually unbiased bases do not exist in many dimensions, this example is perhaps more compelling than the case for SIC-POVMs, since it leads to new necessary conditions which may fail to hold for certain values of $d$.
\end{rmk}

\section*{Acknowledgements}

Part of this work was done at the workshop ``QLA Meets QIT II'', held in Chicago, IL, November 2022. The second author was supported by an AMS-Simons Travel Grant which partially supported this work. Finally, we express our gratitude to the referee for a careful reading of this manuscript and many suggestions which helped improve the exposition.

\printbibliography
\end{document}